\def \eps {\varepsilon}

\documentclass[11pt]{article}

\usepackage{fullpage}
\usepackage{amsfonts}
\usepackage{bbding}
\usepackage{graphicx,float}
\usepackage{amsfonts,amssymb,amsmath,amsthm,amscd}
\usepackage{mathrsfs}
\usepackage{xcolor}
\usepackage{empheq}
\usepackage{adforn}
\usepackage{cancel}
\usepackage{xr-hyper}
\usepackage{xr}
\usepackage{mdframed}
\usepackage{mathdots}
\usepackage{enumerate}
\usepackage{lmodern}
\usepackage{mdframed}
\usepackage{stmaryrd}
\usepackage{blindtext}
\usepackage[all, knot]{xy}
\usepackage{leftidx}
\usepackage{accents}

\definecolor{slightblue}{rgb}{.8, .8, 1}
\definecolor{hair}{RGB}{100,225,190}
\definecolor{ruby}{RGB}{220,50,120}
\definecolor{grass}{RGB}{150,220,110}
\definecolor{ceruleanblue}{rgb}{0.16, 0.32, 0.75}
\definecolor{deepcarmine}{rgb}{0.66, 0.13, 0.24}
\definecolor{otterbrown}{rgb}{0.4, 0.26, 0.13}
\definecolor{sapphire}{rgb}{0.03, 0.15, 0.4}

\usepackage[colorlinks=true,  linkcolor=otterbrown, citecolor=sapphire,  urlcolor=hair!80!black]{hyperref}

\usepackage[T1]{fontenc}

\newtheorem{theorem}{Theorem}[section] \newtheorem{proposition}[theorem]{Proposition}
\newtheorem{lemma}[theorem]{Lemma}

\newtheorem{openquestion}[theorem]{Question}

\theoremstyle{definition} 
\newtheorem{definition}[theorem]{Definition}

\newtheorem{remark}[theorem]{Remark} \numberwithin{equation}{section}
\numberwithin{figure}{section}

\newcommand{\Cb}{\mathbb{C}}

\newcommand{\Eb}{\mathbb{E}}

\newcommand{\Hb}{\mathbb{H}}

\newcommand{\Nb}{\mathbb{N}}

\newcommand{\Pb}{\mathbb{P}}

\newcommand{\Rb}{\mathbb{R}}

\newcommand{\Ub}{\mathbb{U}}

\newcommand{\Zb}{\mathbb{Z}}

\newcommand{\Pf}{\mathbf{P}}

\newcommand{\Bf}{\mathbf{B}}

\newcommand{\Cc}{\mathcal{C}}

\newcommand{\Fc}{\mathcal{F}}

\newcommand{\Lc}{\mathcal{L}}

\newcommand{\Qc}{\mathcal{Q}}

\newcommand{\Sc}{\mathcal{S}}

\usepackage{sidecap,subcaption}
\usepackage[section]{placeins}
\usepackage{wrapfig}
\usepackage{lpic}

\sidecaptionvpos{figure}{c}

\newcommand{\wt}{\widetilde}
\newcommand{\wh}{\widehat}

\begin{document}

\title {Generalized disconnection exponents}
\author{Wei Qian}
\date{}

\maketitle

\begin{abstract}
We introduce and compute the \emph{generalized disconnection exponents} $\eta_\kappa(\beta)$ which depend on $\kappa\in(0,4]$ and another real parameter $\beta$, extending the Brownian disconnection exponents (corresponding to $\kappa=8/3$) computed by Lawler, Schramm and Werner 2001 (conjectured by Duplantier and Kwon 1988). 

For $\kappa\in(8/3,4]$, the generalized disconnection exponents have a physical interpretation in terms of planar Brownian loop-soups with intensity  $c\in (0,1]$, which allows us to obtain the first prediction of the dimension of multiple points on the cluster boundaries of these loop-soups. 
In particular, according to our prediction, the dimension of double points on the cluster boundaries is strictly positive for $c\in(0,1)$ and equal to zero for the critical intensity $c=1$, leading to an interesting open question of whether such points exist for the critical loop-soup.

Our definition of the exponents is based on a certain general version of radial restriction measures that we construct and study. As an important tool, we introduce a new family of radial SLEs depending on $\kappa$ and two additional parameters $\mu, \nu$, that we call \emph{radial hypergeometric SLEs}.\footnote{We will clarify  the term ``hypergeometric SLE''  in Section~\ref{intro:SLE}. The family of chordal hypergeometric SLEs (as well as the notation hSLE) first appeared in \cite{MR3827221}, but this terminology (as well as the notation hSLE) was later inappropriately used  in \cite{MR4072221,Wu1, Wu2, Wu3}.} This is a natural but substantial extension of the family of radial SLE$_\kappa(\rho)s$. 
\bigskip

\end{abstract}

{\setcounter{tocdepth}{2}
  \hypersetup{linkcolor=black}
  \tableofcontents
}

\section{Introduction}
The disconnection and intersection exponents for $n$ independent Brownian motions originated from the observation that certain non-disconnection and non-intersection probabilities for these Brownian motions satisfy a submultiplicativity relation. (This type of argument dates back to \cite{MR0064475}, in the study of self-avoiding walks.) These exponents have later been generalized to make sense for non-integer values of $n$ \cite{MR1742883}. 
The values of the exponents were determined by Lawler, Schramm and Werner in a series of celebrated works \cite{MR1879850,MR1879851, MR1961197} via SLE computations, confirming the conjecture by Duplantier and Kwon \cite{PhysRevLett.61.2514}. 
It turns out that these exponents are related to certain properties of Brownian paths. In particular, the determination of the disconnection exponent for $n=2$ Brownian motions confirmed the famous Mandelbrot conjecture that the outer boundary of planar Brownian motion has Hausdorff dimension $4/3$.

In the pioneering set of works \cite{MR1796962, MR1879850, MR1992830} which exploited conformal invariance and introduced the notion of conformal restriction, it was explained that these exponents are closely related to properties of SLE$_{8/3}$ and SLE$_6$. 
The purpose of the present work is to introduce and compute a generalized version of the disconnection exponents for other values of $\kappa$. We will also informally discuss the generalized intersection exponents at the end.
These generalized exponents are naturally related to geometric properties of Brownian loop-soup clusters.

One of the main difficulties of this work is to find the right definition of such exponents that can be connected to Brownian loop-soups, as well as to find the right family of SLEs to derive these exponents. 
It turns out to be instrumental for us to introduce a new family of \emph{radial hypergeometric SLEs}, as we will explain in Section~\ref{intro:SLE}.

\subsection{Motivation}\label{intro:motivation}

Our work is primarily motivated by the study of planar Brownian loop-soups (introduced in \cite{MR2045953}) with intensity $c\in(0,1]$. 

It is shown in \cite{MR2979861} that this is the range of intensities for which a loop-soup a.s.\ forms infinitely many clusters, and that the outer boundaries of the outermost clusters form a CLE$_\kappa$ where $c$ and $\kappa$ are related by
\begin{align}\label{kappa-c}
c(\kappa)={(6-\kappa)(3\kappa-8)}/{(2\kappa)}.
\end{align}
In particular, the outer boundaries of the clusters have Hausdorff dimension $1+\kappa/8$ \cite{MR2435854}.
Note that the dimension of the outer boundary of a Brownian path is a.s.\ $4/3$,  hence strictly smaller than $1+\kappa/8$ for the range $\kappa\in(8/3,4]$ corresponding by~\eqref{kappa-c} to $c\in(0,1]$.
This implies that a typical point on the boundary of a loop-soup cluster does not belong to any Brownian loop (it is the limit of an infinite chain of Brownian loops).
However, there do exist points that lie on some Brownian loop: If we add an independent Brownian loop into a loop-soup so that it connects several clusters (the law of the new loop-soup is absolutely continuous w.r.t.\ the old one), then some points of this new Brownian loop do belong to the boundary of the new cluster. 

We say that a point is a \emph{simple} (resp.\ \emph{double}, or \emph{$n$-tuple}) point for the loop-soup if it is visited at least once (resp.\ twice, or $n$ times) in total by the Brownian loops in the loop-soup.
Note that the $n$-tuple points for a planar Brownian motion has dimension $2$ for all $n\ge 1$ (this is a nontrivial fact shown in \cite{MR0202193}), hence $n$-tuple points inside a loop-soup should also have the full dimension. 
The original purpose of the present work was to answer the following question.
\begin{openquestion}\label{Q1}
Are there double points on the loop-soup cluster boundaries with intensity $c\in(0,1]$?
\end{openquestion}
From known results, it is not clear what to expect about the answer to Question~\ref{Q1}. Even though we know that Brownian motions have double points on their boundaries \cite{MR2644878}, it does not imply the existence of double points on the loop-soup cluster boundaries, even for arbitrarily small intensity $c$.

Question~\ref{Q1} is motivated by our previous work \cite{Qian-Werner} on the decomposition of Brownian loop-soup clusters. In particular, in a loop-soup cluster with intensity $c=1$, the results of  \cite{Qian-Werner, qian2018} imply that the Brownian loops that touch the boundary of the cluster can be decomposed into a Poisson point process of Brownian excursions. One natural question is to find out how to hook them back into Brownian loops. The resampling property for Markov loop-soups at criticality proved in \cite{MR3618142} gives a heuristic indication that one should be able to exchange the trajectories of loops at random at double points of a critical Brownian loop-soup ($c=1$) without changing its global law. Therefore it is important to answer Question~\ref{Q1} when $c=1$ in order to evaluate  how much randomness is involved in the process of reconstructing the loops from the excursions. For other values $c\in(0,1)$, similar but weaker results have been obtained in \cite{MR3901648}. Ultimately we want to answer the same type of questions for all  $c\in(0,1]$.

It seems to us that the only available way of answering Question~\ref{Q1} is to compute the dimension of these double points explicitly. We plan to carry out this computation in two steps. In the present paper, we will define and compute  the generalized disconnection exponents, and give a mostly heuristic explanation of their relation to the Brownian loop-soups. Then in a forthcoming work \cite{Qian-future}, we will rigorously prove that they are indeed related to the dimension of multiple points on cluster boundaries, according to the usual relation. The second part of this work focuses on a different aspect, and is rather long and technical, hence we decided to separate it from the present article. 

Assuming the usual relation between the exponents and the dimensions, the results of the present article already allow us to predict the Hausdorff dimensions of the simple and double points on the cluster boundaries, and to predict that there are no triple points on the cluster boundaries. 
Note that there was previously no conjectures on the dimensions (or even existence, in general) of these points, since the introduction of Brownian loop-soups fifteen years ago.

In particular, we predict the following interesting phenomenon:
There exist double points on subcritical cluster boundaries ($c\in(0,1)$) and their dimension is strictly positive. However the dimension of  double points on the boundaries of the critical ($c=1$) loop-soup clusters is  zero. 
Moreover, this dimension is strictly decreasing and continuous in $c$. 
Unfortunately, this observation fails to completely answer Question~\ref{Q1}. It still leaves us with the following intriguing open question which seems to require new ideas.
\begin{openquestion} \label{oq}
Are there double points on the critical ($c=1$) loop-soup cluster boundaries?
\end{openquestion}

\subsection{Main result on generalized disconnection exponents}\label{intro:disc}
In this section, we will directly define the generalized disconnection exponents and present our main result about them.
An important part of the work is actually to find the right definition, but we will postpone this explanation to Section~\ref{S:loop-soup}.

Our definition of the generalized disconnection exponents relies on a certain general version of restriction measures that we now define:
Let $\Omega$ be the collection of all  simply connected compact sets $K\subset \overline\Ub$ such that $0\in K$ and $K\cap \partial \Ub=\{1\}$.
Let $\Qc$ be the collection of all compact sets $A\subset\overline\Ub$ such that $\Ub\setminus A$ is simply connected and $0,1\not \in A$. 
For all $A\in\Qc$, let $f_A$ be the conformal map from $\Ub\setminus A$ onto $\Ub$ that leaves $0,1$ fixed.

\begin{definition}[General radial restriction measure]\label{def:gen}
For $\kappa\in(0,4]$ related to $c$ by~\eqref{kappa-c}, a probability measure $\Pb$ on $\Omega$ (or a random set $K$ of law $\Pb$) is said to satisfy radial $\kappa$-restriction with exponents $(\alpha, \beta)\in\Rb^2$, if for all $A\in\Qc$, we have
\begin{align}\label{eq:gen}
\frac{d\Pb(K)}{d\Pb_A(K)}\mathbf{1}_{K\cap A=\emptyset}=\mathbf{1}_{K\cap A=\emptyset}|f_A'(0)|^\alpha f_A'(1)^\beta \exp\left(c(\kappa) \, m_{\Ub}(K,A)\right),
\end{align}
where $\Pb_A=\Pb\circ f_A$ and $m_{\Ub}(K,A)$ is the mass under the Brownian loop measure (defined in \cite{MR2045953}) of all loops in $\Ub$ that intersect both $K$ and $A$.
\end{definition}

Note that general restriction properties of this type involving masses of the Brownian loop measure have been discovered for SLEs (see for example \cite{MR1992830,MR2118865,MR2518970} and see Section~\ref{intro:restriction} for more discussion on restriction measures). 
Definition~\ref{def:gen} is a radial version of this property, with parameters $\alpha,\beta$. When $\kappa=8/3$, we have $c(\kappa)=0$ hence the Brownian loop term vanishes, so they coincide with the standard radial restriction measures studied in \cite{MR3293294}.

We remark that, given $\alpha, \beta\in\Rb$, it remains to be shown that~\eqref{eq:gen} uniquely determines the law of $\Pb$ (this was proved only for $\kappa=8/3$, but we believe it should be the case for all $\kappa\in(0,4]$). In this work (see Section~\ref{intro:restriction}), we will prove the existence of such measures for $(\alpha,\beta)$ in a certain range~\eqref{eq:thm_range}.
For each $\alpha, \beta$ in~\eqref{eq:thm_range}, we will explicitly construct (in Section~\ref{Sec:cons}) a measure which satisfies $\kappa$-restriction with exponents $(\alpha, \beta)$, and we denote it by $\Pf_\kappa^{\alpha,\beta}$. Moreover, if $K$ has law $\Pf_\kappa^{\alpha,\beta}$ with $(\alpha,\beta)$ in the range
\begin{align}\label{range1}
\alpha\le\eta_\kappa(\beta), \quad \beta\ge (6-\kappa)/(2\kappa),
\end{align}
then  (see Theorem~\ref{intro-thm:restriction}) $K$ a.s.\ contains the origin in its interior.

Before giving our definition of the generalized disconnection exponents, we recall that for any simply connected domain $D\subset \Ub$ such that $0\in D$, the conformal radius of $D$ seen from the origin is defined to be $|f'(0)|^{-1}$ where $f$ is any conformal map from $D$ onto $\Ub$ that leaves the origin fixed.

\begin{definition}[Generalized disconnection exponent]\label{intro:def-dis-3}
Fix $\kappa\in(0,4]$. For $\alpha,\beta$ in the range~\eqref{range1},
let $K$ be a radial restriction sample with law $\Pf^{\alpha,\beta}_\kappa$. Let $K_0$ be the connected component containing $0$ of the interior of $K$.
Let $p^R_\kappa(\alpha, \beta)$ be the probability that  the conformal radius of $K_0$ seen from the origin is smaller than $1/R$.
We define $\eta_\kappa(\alpha,\beta)$ to be the exponent such that as $R\to\infty$,
\begin{align}\label{p-R-3}
p^R_\kappa(\alpha,\beta)=R^{-\eta_\kappa(\alpha,\beta)+o(1)}.
\end{align}
For all $\beta\ge (6-\kappa)/(2\kappa)$, define the generalized disconnection exponent $\eta_\kappa(\beta):=\eta_\kappa(0,\beta)$.
\end{definition}

This definition makes sense thanks to our first main result that we now state:
\begin{theorem}\label{main-theorem}
For all $\alpha,\beta$ in the range (\ref{range1}) and for all $R>1$,  we  have that
\begin{align}\label{series}
p^R_\kappa(\alpha,\beta)=\sum_{n=0}^\infty a^n_\kappa(\alpha,\beta)\, R^{-\eta^n_\kappa(\alpha,\beta)},
\end{align}
where $(\eta^n_\kappa(\alpha, \beta))_{n\in \Nb}$ is a positive increasing sequence given by
\begin{align}\label{intro:kappa_ab}
\eta^n_\kappa(\alpha, \beta)=\left(n^2+n-\frac12 \right)\frac{\kappa}{8}-\frac{n-1}{2}-\frac{1}{\kappa}+\frac{\beta}{2}+\left(\frac18\left(n+\frac12 \right)-\frac{1}{4\kappa} \right)\sqrt{16\kappa\beta+(4-\kappa)^2}-\alpha
\end{align}
and
\vspace{-3mm}
\begin{align*}
a^n_\kappa(\alpha,\beta)=\prod_{l=0,l\not=n}^\infty (1-\eta^n_{\kappa}(\alpha,\beta)/\eta^l_\kappa(\alpha,\beta))^{-1}.
\end{align*}
In particular, for all $\kappa\in(0,4]$ the generalized disconnection exponent is given by
\begin{align}\label{thm:eta_kappa}
\eta_{\kappa}(\beta)=\eta_\kappa^0(0,\beta)=\frac{\left(\sqrt{16\beta\kappa+(4-\kappa)^2}- (4-\kappa) \right)^2-4\,(4-\kappa)^2}{32\kappa}.
\end{align}
\end{theorem}
Let us make a few remarks about Theorem~\ref{main-theorem}:
\begin{itemize}

\item If we take $\kappa=8/3$, then we recover the Brownian disconnection exponents established in \cite{MR1879851}:
$$\eta(\beta)=\frac{1}{48} \!\left( \!\left(\sqrt{24\beta+1} -1 \right)^2-4 \right).$$ 

\item Note that (\ref{series}) is an exact series development, not just the asymptotics of the leading term as in~\eqref{p-R-3}.  It would be interesting to know whether the exponents given by~\eqref{intro:kappa_ab} have a physical meaning.

\item The analytic function $\eta_\kappa$ as in~\eqref{thm:eta_kappa} is well-defined for a wider range of $\alpha,\beta$ than~\eqref{range1}, but the geometric interpretation in terms of general restriction measures only holds for $\alpha,\beta$ in~\eqref{range1}.

\item For $\kappa$ and $c$ related by (\ref{kappa-c}), equation \eqref{thm:eta_kappa} can be equivalently written as a function of $c$
\begin{align}\label{eta-c}
\eta_{\kappa(c)}(\beta)=\frac{1}{48}\left(\left(\sqrt{24\beta+1-c}-\sqrt{1-c} \right)^2-4\,(1-c) \right).
\end{align}
Note that the above expression is strictly increasing in $c$ for all $c\in(-\infty,1]$. 

\item  We will explain in Section~\ref{S:loop-soup} that the dimensions of single and double points on the cluster boundaries of a loop-soup with intensity $c\in(0,1]$ are expected to be given by $2-\eta_{\kappa(c)}(2)$ and $2-\eta_{\kappa(c)}(4)$. This is the usual relation that is satisfied in the Brownian case \cite{MR1386292,MR2644878} and we plan to prove it in a future work \cite{Qian-future}. 
As a consequence of this prediction, by~\eqref{eta-c}, these dimensions should be strictly decreasing in $c$. 
In particular, we expect that for all $c<1$,  there exist double points on the cluster boundaries, since  $2-\eta_{\kappa(c)}(4)>0$. 
However, for $c=1$, we expect the dimension of double points on the cluster boundaries to be exactly equal to zero, since $\eta_4(4)=2.$
This leads to Question~\ref{oq}.

\item Finally, we remark that in \cite{MR1879851}, the so-called annulus crossing exponents for radial SLE$_\kappa$ were defined and computed.
It would be interesting to investigate their relation to the disconnection exponents considered in the present paper.
\end{itemize}

\subsection{Main result on general radial restriction measures}\label{intro:restriction}
In this section, we will present our results on the general radial restriction measures defined in Definition~\ref{def:gen}.

As we have mentioned, restriction  properties of the type~\eqref{eq:gen} which involve masses of the Brownian loop measure have been discovered for SLE curves, see for example \cite{MR1992830,MR2118865, MR2518970}.
The particular case of $\kappa=8/3$ (where the Brownian loop term vanishes)  was extensively studied in  \cite{MR1992830} where they focused on the natural and important chordal case. This was later followed by the works \cite{MR3293294} on the radial case and \cite{MR3827221} on the trichordal case. 
Definition~\ref{def:gen}, when $\kappa=8/3$, corresponds to the standard radial restriction measures \cite{MR3293294}.
However, the generalized family given by Definition~\ref{def:gen} has not been studied before.

The uniqueness of measures that satisfy Definition~\ref{def:gen} for each $(\alpha, \beta)$ remains to be proved (except for $\kappa=8/3$). In this work, we will prove the existence of such measures for a certain range of $(\alpha,\beta)$, by constructing these measures. This is enough for the purpose of the present work. More precisely, we have the following result.


\vspace{3mm}
\begin{figure}[h!]
\centering
\includegraphics[width=0.9\textwidth]{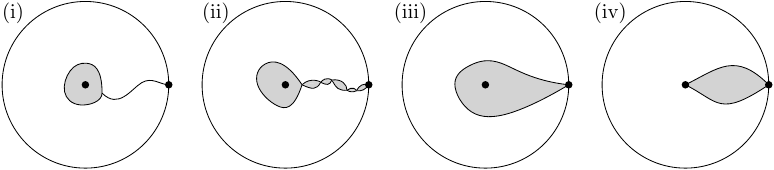}
\caption{Geometric properties of some of the cases in Theorem~\ref{intro-thm:restriction}.}
\label{fig:restriction}
\end{figure}
\vspace{-2mm}

\begin{theorem}\label{intro-thm:restriction}
For any $\kappa\in(0,4]$, and $\alpha, \beta$ in the following range where $\eta_\kappa$ is the $\kappa$-disconnection exponent given by~\eqref{thm:eta_kappa}
\begin{align}\label{eq:thm_range}
\alpha\le\eta_\kappa(\beta), \quad \beta\ge (6-\kappa)/(2\kappa),
\end{align}
there exists a measure $\Pf_\kappa^{\alpha, \beta}$ which satisfies  radial $\kappa$-restriction as in Definition~\ref{def:gen} with parameters $(\alpha,\beta)$.
Moreover, if $K$ is a sample with law $\Pf_\kappa^{\alpha, \beta}$, then it satisfies the following geometric properties:
\begin{enumerate}[(i)]
\item If $\beta=(6-\kappa)/(2\kappa)$, then $K$ is a.s.\ a simple curve in the neighborhood of $1$, see Figure~\ref{fig:restriction}(i). If in addition $\alpha=\eta_\kappa(\beta)$, then $K$ is just a radial SLE$_\kappa$ curve (which is a simple curve from $1$ to $0$).
\item If  $\beta\in ((6-\kappa)/(2\kappa),  (12-\kappa)(\kappa+4)/(16\kappa))$, then $K$ is a.s.\ not a simple curve, but has cut-points, see Figure~\ref{fig:restriction}(ii).
\item If $\beta\ge  (12-\kappa)(\kappa+4)/(16\kappa)$, then $K$ a.s.\ does not have cut-points, see Figure~\ref{fig:restriction}(iii).
\item If $\alpha=\eta_\kappa(\beta)$, then $K$ a.s.\ has the origin on its boundary, see Figure~\ref{fig:restriction}(iv).
\end{enumerate}
\end{theorem}

We believe that \eqref{eq:thm_range} is the maximal range for which radial $\kappa$-restriction measures with parameters $(\alpha, \beta)$ exist. This was proved to be the case for $\kappa=8/3$ \cite{MR3293294}, but it remains to be proved for other $\kappa\in(0,4]$.

We recover the results of \cite{MR3293294} when $\kappa=8/3$. However, our method differs from the one used in \cite{MR3293294}: 
\begin{itemize}
\item In \cite{MR3293294},  to construct radial restriction measures with parameters $\alpha,\beta$ such that $\alpha<\eta(\beta)$, Wu takes the union of a radial restriction measure of parameter $(\eta(\beta), \beta)$ with an independent Poisson point process of Brownian loops surrounding the origin of intensity $\eta(\beta)-\alpha$ (the filled set of which satisfies radial restriction with parameters $(\alpha-\eta(\beta), 0)$). 
\item The method above cannot be applied to $\kappa\not=8/3$, since the union of two independent $\kappa$-restriction measures no longer satisfies $\kappa$-restriction. Therefore, we have taken another route which is to construct the outer boundaries of the restriction measures directly using a certain variant of radial SLE curves that we will introduce in Section~\ref{intro:SLE}. In particular, our result applied to $\kappa=8/3$ also gives an SLE description of the outer boundaries of the standard radial restriction measures constructed in \cite{MR3293294}.
\end{itemize}

Let us mention that in Section~\ref{sec:further}, we will also construct the chordal and trichordal versions of general restriction measures. The construction in these two cases is much more straightforward than in the radial case, hence we will put it as a side remark at the end.

\subsection{Main results on radial hypergeometric SLE}\label{intro:SLE}
In this section, we will present our main results on a new family of radial SLEs, that we call radial hypergeometric SLE. This family depends on $\kappa$ and two additional parameters $\mu$ and $\nu$. 
It is a natural but substantial generalization of the family of radial SLE$_\kappa(\rho)$s.
It is the main tool to construct the general radial restriction measures and to derive  the exponents in Theorem~\ref{main-theorem}.

Since the invention of SLE by Schramm in \cite{MR1776084}, many variants of SLE have been introduced. 
The most common variants are chordal and radial SLE$_\kappa(\rho)$s (see \cite{MR1992830}). These SLEs depend on one additional marked point on the boundary, apart from the starting and ending points of the curve.
In the chordal setting, more complicated variants of SLE involving hypergeometric functions have appeared in the study of \emph{multiple SLEs} (it is a system of $n$ curves so that each curve has the law of an SLE in the complement of the other $n-1$ curves). There are numerous works on this subject, see for example \cite{MR2358649, MR2187598, MR2310306}. Hypergeometric functions have also appeared in the driving function of \emph{intermediate SLEs}, which were introduced by Zhan in \cite{MR2646499}.
This family depends on $\kappa$ and an additional parameter $\rho$. These SLEs depend on two marked points on the boundary (in addition to the starting and ending points of the curve). They were shown in \cite{MR2646499} to be the time reversal of chordal SLE$_\kappa(\rho)$s as one of its marked points tends to the starting point.

The family of (chordal) hypergeometric SLEs considered in \cite{MR3827221} depends on $\kappa$ and three additional parameters, and are denoted there by hSLE$_\kappa (\mu,\nu, \lambda)$. These SLEs, like intermediate SLEs, also depend on two marked points on the boundary (in addition to the starting and ending points of the curve). The family of chordal hSLEs strictly contains the family of intermediate SLEs, since it has two more parameters. The family of intermediate SLEs further contains as a special case  the marginal law of one SLE in the configuration of multiple SLEs with $n=2$.  The family of hypergeometric SLEs also contains the family of chordal SLE$_\kappa(\rho_1, \rho_2)$ with two marked points on the boundary.

The same terminology of hypergeometric SLE (and the same notation hSLE) was later used in \cite{MR4072221} and in a series of preprints \cite{Wu1, Wu2, Wu3}. 
However, the family considered there is identical to the intermediate SLEs considered in \cite{MR2646499}, except that Wu defined this family for $\kappa\in(0,8)$ (Zhan only considered the case $\kappa\in(0,4]$). 
In \cite{MR4072221}, the relation of this family to previously-known families of SLEs was not explained, apart from claiming that the SLEs in  \cite{MR2646499} was ``a particular case'' of the family considered in \cite{MR4072221}, and the SLEs in \cite{MR3827221} are ``different'' from those in \cite{MR4072221} (without explaining why, see \cite[end of Section 1]{MR4072221}).  
In \cite[two lines above Section 1.1]{Wu3}, the authors claimed that the family of hypergeometric SLEs was introduced by Wu in \cite{MR4072221}.\footnote{We wish to make this clarification here, because the inappropriate use of this terminology has caused confusion. In particular, some people seemed to believe that the hypergeometric SLEs considered in \cite{MR3827221} were also nothing else but intermediate SLEs.}
Let us also mention that many properties of intermediate SLEs, such as reversibility and a certain ``coupling''  (developed in \cite{MR2435856,MR2439609}, and closely related to ``commutation relations'' in~\cite{MR2358649}), were established in  \cite{MR2646499} (for $\kappa\in(0,4]$), but results of similar type were proved in \cite{MR4072221} (for $\kappa\in(0,8)$) with no reference to  \cite{MR2646499}.

\begin{figure}[h!]
\centering
\includegraphics[width=0.8\textwidth]{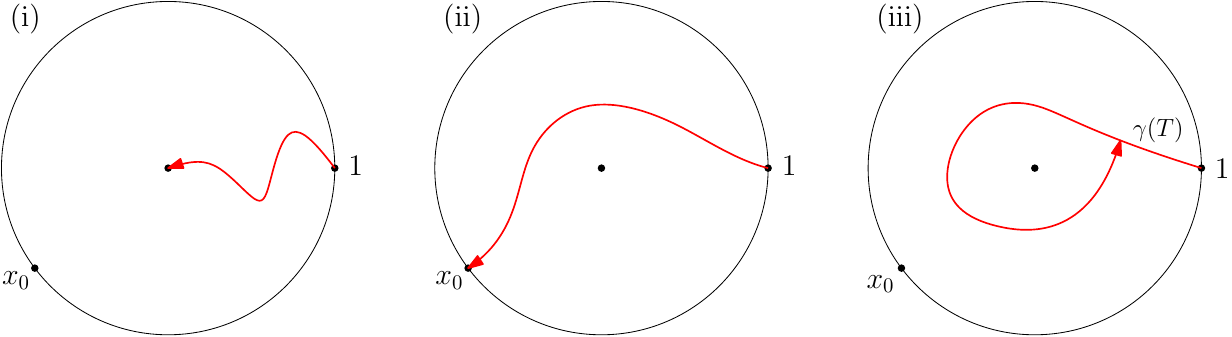}
\caption{Three possible behaviors of radial hSLE}
\label{fig:hsle1}
\end{figure}

In the present work, we will introduce the radial counterpart of the hypergeometric SLEs defined in \cite{MR3827221}.
To our knowledge, the only previously-known sub-family of the radial hSLEs is radial SLE$_\kappa(\rho)$.
The radial hSLEs are SLE curves in a simply connected domain starting from one boundary point and targeting  one interior point with one additional  marked point on the boundary. Their driving function is a Brownian motion with a drift term involving hypergeometric functions. We will denote them as hSLE$_\kappa(\nu, \mu)$.

The definition of the radial hSLEs  will involve quite complicated equations, so we will postpone this to Section~\ref{ShSLE}. We will also state precise results about radial hSLEs in Section~\ref{ShSLE}.
Here in the introduction, we would mention some of their interesting geometric properties, see Figure~\ref{fig:hsle1} (stated in Proposition~\ref{prop:hsle_geom}). In some cases, the radial hSLE curve forms a particular shape of a loop surrounding the origin, despite being a simple curve elsewhere. This specific behavior is, to our knowledge, different from any other known types of SLE's. 

More precisely, for $\kappa\in(0,4]$, consider a radial hSLE$_\kappa(\mu, \nu)$ in $\Ub$ starting at $1$, targeting $0$, with a marked point $x_0$. Depending on its parameters $\mu$ and $\nu$, it can have three different behaviors: 
\begin{enumerate}[(i)]
\item It is a.s.\ a simple curve from $1$ to $0$, in which case it is just a SLE$_\kappa(\rho)$. 
\item It is a.s.\ a simple curve from $1$ to $x_0$ which almost surely leaves $0$ on its left.
\item It a.s.\ surrounds the origin counterclockwisely before coming back and hitting itself or the counterclockwise arc from $x_0$ to $1$. In particular, if we place $x_0$ immediately below $1$, then the curve will a.s.\ come back hitting itself, forming a loop around the origin.
\end{enumerate}
The case (i) is in some sense the degenerate limiting case, and consists of a two-parameter sub-family of the three-parameter family of radial hSLEs. However, the cases (ii) and (iii) are generic and both depend on three parameters.

A major difference between the generic radial hSLEs (cases (ii) and (iii))  with radial SLE$_\kappa(\rho)$s (case (i)) is that the generic radial hSLEs are \emph{not} left-right symmetric: Let $\Sc$ be the mirror symmetry map from the unit disk onto itself with respect to the real axis. Let $\gamma$ a hSLE from $1$ to $0$ with a marked point $x_0$. Then $\Sc(\gamma)$ is \emph{not} equal in distribution to a hSLE from $1$ to $0$ with the marked point $\Sc(x_0)$, unless $\gamma$ is an SLE$_\kappa(\rho)$.

Let us point out an interesting special example of radial hSLE which illustrates this asymmetry: a chordal SLE$_\kappa$ conditioned to pass to the right of some fixed interior point, seen as a radial SLE towards that point. This example belongs to the case (ii) above. The probability that a chordal SLE$_\kappa$ passes to the right of a fixed interior point has been computed by Schramm in \cite{MR1871700}. The fact that such a conditioned chordal SLE$_\kappa$ is an instance of radial hSLE can be deduced (with some extra work) from the computations in \cite{MR1871700}. We leave this proof to the interested reader.

\subsection{Outline of the paper}
The present paper will be organized as follows.

In Section~\ref{S:loop-soup}, we will recall some background and explain the heuristic relation between the generalized disconnection exponents and the Brownian loop-soups.
We emphasize that this section is \emph{not} needed to understand the rest of this article. However, it explains the physical meaning of the present work and its relation to our future work \cite{Qian-future}.

In Section~\ref{ShSLE}, we will define the three-parameter family of radial hypergeometric SLEs. 
An important part of the work was in fact to find these SLEs. However, to keep the presentation concise,  we will in practice directly give the equations of these SLEs and then analyse their properties. 

In Section~\ref{Sec:cons}, we will construct the radial general restriction measures using the radial hSLEs, proving Theorem~\ref{intro-thm:restriction}. The strategy of this part is similar to that of \cite{MR1992830}.

In Section~\ref{sec-disconnection}, we will prove Theorem~\ref{main-theorem}, and consequently deduce the values of the generalized disconnection exponents. We will mainly analyse an explicit diffusion process associated to the hSLEs, using classical diffusion theory and a result of \cite{MR576891}.

In Section~\ref{sec:further}, we will make some remarks on the generalization of Brownian intersection exponents, and construct the chordal and trichordal versions of general restriction measures.

Finally, we will postpone some lengthy computations for SLE to Appendix~\ref{A1}.

\subsection*{Acknowledgements}
This work was done while the author was affiliated to the University of Cambridge. The author acknowledges the support by EPSRC grant EP/L018896/1 and a JRF of Churchill college. 
We thank  Pierre Nolin and Wendelin Werner for useful comments on an earlier version of this article.
We thank Bertrand Duplantier for pointing out the connection between the exponents \eqref{tilde-xi-c}, \eqref{xi-c}  and \cite{MR2112128, MR2581884}. We also thank the anonymous referee for reading this article carefully and for many valuable comments.

\section{Relation between the generalized disconnection exponents and loop-soups}\label{S:loop-soup}
In this section, we are going to explain the physical meaning behind the generalized disconnection exponents given by Definition~\ref{intro:def-dis-3}. 
In particular, we will recall some related background and point out the heuristic relation (which we plan to rigorously establish in \cite{Qian-future})  between the generalized disconnection exponents with parameter $\kappa\in(8/3,4]$ and the dimensions of simple and double points on the cluster boundaries in a Brownian loop-soup of intensity $c(\kappa)$.

\subsection{Background on the Brownian disconnection exponents}\label{sec:background}
The Brownian disconnection exponents are defined as follows: Consider $n$ independent Brownian motions started at $n$ uniformly chosen points on the unit circle $\partial \Ub$, and stopped upon reaching the boundary $\partial B(0,R)$ of the ball with radius $R$ around $0$.  
We say that a set $K\subset \Cb$ disconnects $0$ from $\infty$ if $0$ and $\infty$ are not in the same connected component of $\Cb\setminus K$.
Let $p^R_n$ be the probability that the union of the $n$ stopped Brownian motions does not disconnect $0$ from $\infty$.
Then by scale invariance and the Markov property of Brownian motions, it is easy to see that for any $R,S>1$, we have
\begin{align*}
p_n^{RS} \le p_n^R \cdot p_n^S.
\end{align*}
This implies, by subadditivity, that $\log p_n^R / \log R$ converges as $R$ goes to $\infty$. The disconnection exponent $\eta(n)$ for $n\in\Nb^*$ (where $\Nb^*$ denotes the set of all positive integers) is then defined to be
\begin{align}\label{dis-br}
\eta(n)= - \lim_{R\to \infty} \log p_n^R / \log R.
\end{align} 
The exponent $\eta$ was then further extended to non-integer arguments \cite{MR1742883}, leading to a continuous one-parameter family $\eta(\beta)$.
The value of $\eta(\beta)$ was determined in \cite{MR1879851} (also see the related works \cite{MR1742883, MR1879850, MR1899232, MR1961197}). 

The Brownian disconnection exponents are directly related to the dimension of the outer boundary of a Brownian motion or of the double points on the outer boundary. Here is a heuristic explanation: A point is on the boundary of a Brownian motion  means that the past and future parts of the Brownion motion (which are two conditionally independent Brownian motions) do not disconnect it from $\infty$.
Therefore (see \cite{MR1386292} for more details), the Hausdorff dimension of the Brownian frontier  is $2-\eta(2)$. Similarly, if a point on the Brownian frontier is visited twice, then there need to be four Brownian arms starting from that point, hence the dimension of double points on the Brownian frontier is 
 $2-\eta(4)$ (\cite{MR2644878}). Finally, there is no triple points on the Brownian frontier because $2-\eta(6)<0$ (\cite{MR1387629}).

\subsection{Relation between Brownian motions and restriction measures}\label{ss:res}
It is not immediately clear that the exponents considered in Section~\ref{sec:background} should coincide with the ones given by Definition~\ref{intro:def-dis-3} (when $\kappa=8/3$), but Theorem~\ref{main-theorem} implies that they have the same values. In this section, we are going to give heuristic arguments which explain why it should be the case.

It was pointed out in \cite{MR1992830} that Brownian motions satisfy a certain \emph{conformal restriction} property, which refers to the combination of conformal invariance and the following property: Loosely speaking, if we condition a certain Brownian trace in a given domain to stay in a subdomain, then its law is the same as if we directly sample a Brownian motion in this subdomain.
Restriction measures have first been studied by Lawler, Schramm and Werner in \cite{MR1992830} who focused on the chordal case. This was later followed by the study of the radial case  \cite{MR3293294} and of the trichordal case \cite{MR3827221}. In all three cases, these measures are characterized by a few (respectively one, two, or three) real parameters.

More precisely, let $\Bf$ be a Brownian motion started from the origin and stopped upon hitting $\partial \Ub$ at time $T$.
Let $J$ be the rotated (around the origin) trace of $\Bf([0,T])$ so that $J\cap \partial \Ub=\{1\}$. Let $K$ be the filling of $J$, namely the complement of the unbounded connected component of $\Cb\setminus J$.
Then $K$ satisfies the radial restriction property  with parameters $(0,1)$ at the marked points $0$ and $1$ in $\Ub$. In other words, its law is determined by the following property: For all $A\in\Qc$, we have
\begin{align*}
\Pb(K\cap A=\emptyset)=f_A'(1),
\end{align*}
where (we recall that) $\Qc$ and $f_A$ have been defined just before Definition~\ref{def:gen}.

This suggests that we can also relate the Brownian disconnection exponents to radial restriction measures.
For example, let $\Bf$ be a Brownian motion starting at $1$ and exiting $\partial B(0,R)$ at time $T$. The probability $p_1^R$ changes at most by a multiplicative constant if we condition on $\Bf(T)=R$. (To deduce this fact, one can for example decompose the trajectory of $\Bf$ at the last time $\tau$ that it reaches $\partial B(0,R/2)$, see \cite{MR2045953}. Then  $B([\tau, T])$ has a positive probability of not contributing to the event of disconnection, uniformly on the position of $B(T)$).
We then condition on $\Bf(T)=R$ and let  $\varphi$ be the conformal map from $B(0,R)$ onto $\Ub$ which maps $1, R$ to $0,1$. 
By conformal invariance of Brownian motion, the filling of $\varphi(\Bf([0,T]))$ is a radial restriction sample with parameters $(0,1)$ at the marked points $0$ and $1$ in $\Ub$. 
Moreover, $\varphi(0)\in \Rb^-$ and $|\varphi(0)|$ is of order $1/R$.
This leads to an alternative and equivalent definition of $\eta(1)$: Let $K$ be a radial restriction measure in $\Ub$ with parameters $(0,1)$ at the marked points $0$ and $1$. Let $\wt p^R$ be the probability that $K$ does not disconnect $-1/R$ from $\infty$. Then define
\begin{align}\label{dis-res}
\eta(1)= - \lim_{R\to \infty} \log \wt p^R / \log R.
\end{align}

For $n\in \Nb_{\ge 2}$,  the union of $n$ independent Brownian motions started at $0$ and conditioned to exit $\partial \Ub$ at $1$ (which can be achieved by rotation) satisfies radial restriction with parameters $(0,n)$.
This suggests that $\eta(n)$ could be similarly defined  in terms of a radial restriction measure of parameters $(0,n)$. However, the $n$ Brownian motions in the original definition (given in Section~\ref{sec:background}) all start at different points on the unit circle. It involves some additional difficulty to establish the equivalence between the original definition and a definition using restriction measures, because it turns out to be more complicated to estimate the disconnection probability when we change the positions of the starting points of the Brownian motions than when we change their exiting points.
\smallskip

We remark that Definition~\ref{intro:def-dis-3} (applied to $\kappa=8/3$ and $n=1$) 
is very similar to the definition given by~\eqref{dis-res}, except that we considered in Definition~\ref{intro:def-dis-3} the event that the conformal radius of $K$ is $\le 1/R$, instead of the event that $K$ does not disconnect $-1/R$ from $\infty$.
However, these two events are within constant multiplicative factors from each other, hence Definition~\ref{intro:def-dis-3} and~\eqref{dis-res} indeed yield the same exponent. This can for example be proven by Koebe $1/4$ theorem and basic estimates on SLE, but we do not plan to carry it out here.
One reason in favour of considering the probability that the conformal radius of $K$ is small, rather than $\wt p^R$, is that the first quantity is invariant w.r.t.\ rotation of $K$.

\subsection{On the generalized disconnection exponents}
In this section, we will comment on Definition~\ref{intro:def-dis-3}.
Before that, we first introduce another way of defining the generalized disconnection exponents which only involves loop-soups and Brownian motions. This definition is analogous to the  definition of the Brownian disconnection exponents in Section~\ref{sec:background}. 

\begin{figure}[h!]
\centering
\includegraphics[width=0.38\textwidth]{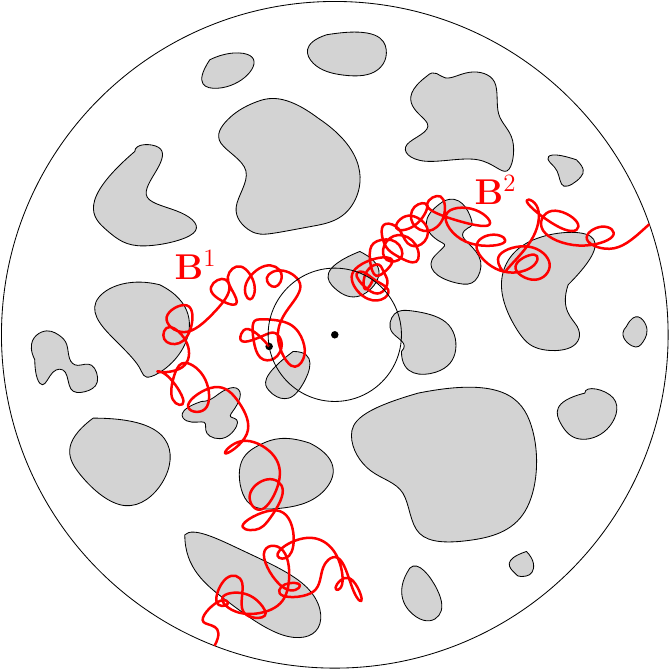}
\caption{We illustrate the event (for $n=2$) that the union of two Brownian motions together with the loop-soup $\Lambda_{R,1}$ does not disconnect the origin from $\infty$.
The two Brownian motions are depicted in red and the filled ourtermost clusters of $\Lambda_{R,1}$ are depicted in grey.}
\label{fig:loop-soup}
\end{figure}

Throughout, we fix $\kappa\in(8/3, 4]$ and $c\in(0,1]$ related to each other by~\eqref{kappa-c}.
Let $\Lambda$ be a Brownian loop-soup of intensity $c$ in the whole plane.
For all $R>S$, let $\Lambda_{R,S}$ be the collection of all the loops in $\Lambda$ which are contained in $B(0,R)$, but not contained in $B(0,S)$. 
Consider $n$ independent Brownian motions started at $n$ uniformly chosen points on $\partial B(0,S)$ and stopped upon reaching  $\partial B(0,R)$. 
Let $p(c,n, R, S)$ be the probability that the union of the $n$ stopped Brownian motions together with all the loops in $\Lambda_{R,S}$ does not disconnect $0$ from $\infty$. See Figure~\ref{fig:loop-soup}.
By scale invariance and the Markov property of the Brownian motions and the loop-soup, we have that for all $R_1, R_2>1$,
\begin{align*}
p(c,n,R_1R_2, 1) \le p(c, n, R_1, 1) \, p(c,n, R_2, 1).
\end{align*}
By subadditivity, this implies that $\log p(c,n,R,1)/ \log R$ converges as $R\to\infty$.
We can then define the generalized disconnection exponent of $n$ Brownian motions in a loop-soup of intensity $c$ to be the following
\begin{align}\label{lim:gen}
-\lim_{R\to \infty} \log p(c,n, R,1) / \log R.
\end{align}

This way of defining the disconnection exponents makes it intuitive that they should be related to the dimensions of simple and double points on the boundaries of loop-soup clusters, since one can apply almost the same reasoning as for the standard Brownian case (as we have explained at end of Section~\ref{sec:background}).
Nevertheless, rigorously deriving these dimensions is rather technical (similar to \cite{MR1386292,MR2644878}) and will be done in a separate work \cite{Qian-future}.

Similarly to Section~\ref{ss:res}, we will now relate these generalized exponents to general restriction measures. 
The relation comes from the following observation, which is the radial analogue of the chordal result in  \cite{MR3035764} (and can be proven using similar ideas):
\begin{lemma}\label{lem:obs}
Let $K$ be a standard radial restriction sample in $\Ub$ with marked points $0,1$ and exponents $(\alpha, \beta)$. Let $\Cc(K)$ be the union of $K$ with all the clusters that it intersects in an independent loop-soup of intensity $c(\kappa)$ in $\Ub$.
Then the filling of $\Cc(K)$ satisfies radial $\kappa$-restriction of the type~\eqref{eq:gen} with exponents $(\alpha, \beta)$.
\end{lemma}

This then explains the reason why we defined the generalized exponents in Definition~\ref{intro:def-dis-3} using general radial restriction measures.

However, as in the Brownian case, except for $n=1$, it is not straightforward to establish the equivalence between the definition~\eqref{lim:gen} and Definition~\ref{intro:def-dis-3} (when both of them are defined, i.e., when the argument is $n\in\Nb^*$).
We also plan to tackle this difficulty in \cite{Qian-future}.

Finally, let us emphasize that Definition~\ref{intro:def-dis-3} via the general restriction measures has the advantage of being valid for non-integer arguments and all $\kappa\in(0,4]$. 
Note that the loop-soup interpretation (and also Lemma~\ref{lem:obs}) only makes sense for $\kappa\in(8/3,4]$.

\section{Radial hypergeometric SLE}\label{ShSLE}
In this section, we will define the radial hypergeometric SLEs (hSLEs) and analyse their geometric properties. 

In reality,  an important part of the work is to first find the right definition for these SLEs. However, to keep the presentation simple, we will hide this step and directly give their definition. Their driving functions are given by complicated functions and we will explain the reason for our definition later in Remark~\ref{remark} and Appendix~\ref{A3}. 

\subsection{The Loewner equation}
Let us first recall that radial Loewner evolution $K_t$ in $\overline\Ub$ started from $1$ and targeting the origin can be parametrized by the following Loewner equation
\begin{align}\label{radial eq}
\partial_t g_t(z) =g_t(z) \frac{e^{iW_t}+g_t(z)}{e^{iW_t}-g_t(z)},\quad g_0(z)=z,
\end{align}
in a way that $g_t$ is the conformal map from $\Ub\setminus K_t$ onto $\Ub$ such that $g_t(0)=0$ and  $g_t'(0)=\exp(-t)$.

Schramm invented in \cite{MR1776084} the SLE processes by using as an input driving functions $W_t$ that are Brownian motions.
In particular, a radial SLE$_\kappa$ is generated by $W_t=\sqrt{\kappa}B_t$ in (\ref{radial eq}), where $B_t$ is a standard Brownian motion. Recall that for $\kappa\in(0,4]$, SLE$_\kappa$ is a.s.\ a simple curve.

There is a common variant of radial SLE$_\kappa$ which is called radial SLE$_\kappa(\rho)$. They are random curves in $\overline\Ub$ started from $1$ targeting the origin with an additional marked point  $x_0:=\exp(-2\theta_0 i)$ for $\theta_0\in(0,\pi)$.
They are generated by \eqref{radial eq} with driving function $W_t$ which is the unique solution of the following equations where $\theta_t= (W_t-V_t)/2$
\begin{align*}
& d W_t =\sqrt{\kappa} d B_t +\frac{\rho}{2} \cot\left(\theta_t \right) dt\\
& d V_t=-\cot \left(\theta_t \right) dt, \quad V_0=-2\theta_0.
\end{align*}
Let $x_t:=g_t(x_0)$. Note that $x_t=\exp(i V_t )$.
\smallskip

In the following, we aim to introduce a new variant of radial SLE processes: They are also random curves in $\overline\Ub$ started from $1$ targeting the origin with one marked point  $x_0:=\exp(-2\theta_0 i)$ for $\theta_0\in(0,\pi)$. 
These curves will depend on $\kappa$ and two real parameters $\mu$ and $\nu$ and we denote them by hSLE$_\kappa(\mu,\nu)$.

We will first recall some preliminaries on hypergeometric functions in Section~\ref{ss:hyp}, then define in Section \ref{sec:G} some function $G$ (depending on $\kappa,\mu$ and $\nu$) in terms of hypergeometric functions. Then, we will choose the driving function $W_t$ of such an hSLE process to be the solution of the following equations where $\theta_t= (W_t-V_t)/2$
\begin{equation}\label{driving-function}
\begin{split}
& d W_t=\sqrt{\kappa} dB_t+\frac{\kappa}{2}\frac{G'(\theta_t)}{G(\theta_t)} dt\\
& d V_t=-\cot (\theta_t) dt, \quad V_0=-2\theta_0.
\end{split}
\end{equation}
Finally, in Section \ref{S:geom}, we will discuss the basic geometric properties of the hSLEs defined above, as the parameters vary in a certain range.

\subsection{Preliminaries on hypergeometric functions}\label{ss:hyp}
In this section, we will give some preliminaries on hypergeometric functions.

For all $a,b\in\Cb$ and $c\in\Rb\setminus\Zb_-$ (where $\Zb_-$ is the set of all non-positive integers), the hypergeometric function $_2F_1(a,b;c;\cdot)$ is defined for all $|z|< 1$  by the power series
\begin{align}\label{eq:series}
_2F_1(a,b;c;z)=\sum_{n=0}^\infty \frac{(a)_n (b)_n}{(c)_n} \frac{z^n}{n!},
\end{align}
where $(x)_0=1$ and $(x)_n=x(x+1)\cdots (x+n-1)$ for all $n>0$.
It can then be analytically extended to $\Cb\setminus[1,\infty)$ and is a particular solution of Euler's hypergeometric differential equation
\begin{align}
z(1-z) u''(z)+\left(c-(a+b+1)z \right) u'(z) -ab\, u(z)=0.
\end{align}

If $c-a-b$ is not an integer, then for all $z> 0$, we have (see for example \cite{MR1225604})
\begin{equation}\label{hyper-identity}
\begin{split}
_2F_1(a,b;c;z)=&\frac{\Gamma(c)\Gamma(c-a-b)}{\Gamma(c-a)\Gamma(c-b)} z^{-a}{_2F}_1\left(a,a-c+1;a+b-c+1;1-\frac{1}{z} \right)\\
&+\frac{\Gamma(c)\Gamma(a+b-c)}{\Gamma(a)\Gamma(b)}(1-z)^{c-a-b}z^{a-c}{_2F}_1\left(c-a, 1-a; c-a-b+1; 1-\frac{1}{z} \right),
\end{split}
\end{equation}
and for all $z<0$, we have
\begin{equation}\label{hyper-identity2}
\begin{split}
_2F_1(a,b;c;z)=&\frac{\Gamma(c)\Gamma(b-a)}{\Gamma(b)\Gamma(c-a)} (-z)^{-a}{_2F}_1\left(a,a-c+1;a-b+1;\frac{1}{z} \right)\\
&+\frac{\Gamma(c)\Gamma(a-b)}{\Gamma(a)\Gamma(c-b)}(-z)^{-b}{_2F}_1\left(b, b-c+1; -a+b+1; \frac{1}{z} \right).
\end{split}
\end{equation}

\subsection{The function $G$: definition and properties}\label{sec:G}
In this section, we are going to define the function $G$ and determine some of its basic properties.

Throughout, $G$ depends only on the parameters $\kappa$ and $\mu, \nu$. To simplify the notations, we  define some auxiliary parameters:
\begin{equation}\label{eq:parameters}
\left. 
\begin{split}
&q_1(\kappa, \mu)=\frac{1}{2\kappa}\sqrt{16\kappa \mu +(4-\kappa)^2}, \quad q_2(\kappa, \nu)=\frac{1}{4\kappa}\sqrt{16\kappa \nu +(4-\kappa)^2},\\
&a=\frac14+q_1(\kappa,\mu)+q_2(\kappa,\nu),\quad
b=\frac14-q_1(\kappa,\mu)+q_2(\kappa,\nu),\quad
c=1+2 q_2(\kappa,\nu),\\
&d=-\frac1\kappa+\frac14+q_2(\kappa,\nu), \quad e=2\mu-\frac{(6-\kappa)(\kappa-2)}{8\kappa}.
\end{split}
\qquad\right|
\end{equation}
We will restrict ourselves to $\kappa\in(0,4]$ and $\mu, \nu\in\Rb$ and in the range
\begin{align}\label{eq:range}
(i)\, \nu\ge 0, \quad (ii)\, \mu \le -\frac{\kappa}{32} +\frac{\nu}{4} -\frac{3}{4\kappa} + \frac38 +\frac{1}{32} \sqrt{16\kappa \nu +(4-\kappa)^2}.
\end{align}
Note that (ii) is equivalent to $b\ge 0$ if $\mu\ge -(4-\kappa)^2/ (16\kappa)$. 
For $\mu< -(4-\kappa)^2/ (16\kappa)$, we interpret the square root in $q_1(\kappa, \mu)$ as 
$$q_1(\kappa,\mu)=\frac{i}{2\kappa}\sqrt{-16\kappa \mu -(4-\kappa)^2}.$$
Our definition of hSLE$(\mu,\nu)$ actually works for a slightly larger range of $(\mu,\nu)$, but will only correspond to general restriction measures if $(\mu,\nu)$ is in~\eqref{eq:range}.

\begin{definition}\label{def:G}
Let $G$ be the function that maps $\theta \in(0, \pi)$ to
\begin{equation}\label{def1}
\begin{split}
G(\theta):=&\frac{\Gamma( c)\Gamma(1/2)}{\Gamma( c- a)\Gamma( c- b)} \left(1+(\cot\theta)^2\right)^{ a-{d}}  {_2}F_1\left( a,  a- c+1; \frac12; -(\cot\theta)^2\right)\\
&+\frac{\Gamma(c)\Gamma(-1/2)}{\Gamma(a)\Gamma( b)} \cot\theta \left(1+(\cot\theta)^2\right)^{ b-{d}}{_2}F_1\left( c-a,1- a;\frac32;-(\cot\theta)^2  \right).
\end{split}
\end{equation}
\end{definition}
Let us make a few remarks about the definition.
\begin{itemize}
\item The function $G$ is well defined and analytic on $(0,\pi)$, since $-(\cot\theta)^2<0$ and both $1/2$ and $3/2$ are not in $\Zb_-$.
By (\ref{hyper-identity}) and the fact that $a+b+1/2=c$, for $\theta\in(0, \pi/2)$, $G$ can be simplified to 
\begin{align}\label{simplified-G}
G(\theta)= (\sin\theta)^{2{d}} {_2F_1} (a,b;c; (\sin\theta)^2).
\end{align}
For $\theta\in[\pi/2,\pi)$, $G$ is the analytic extension of (\ref{simplified-G}). Even though we have defined $G$ only on the real interval $(0, \pi)$, it can be viewed as an analytic function in the complex plane, defined in a small neighborhood of $(0, \pi)$ (This is because~\eqref{def1} only has singularities at $\theta$ where $\cot(\theta)=\pm i$ or $\infty$, so these singularities are away from the real interval $(0,\pi)$. At any point  $\theta \in (0, \pi)$, the Taylor series of $G$ has a positive radius of convergence.) This is why we can obtain $G$ as the unique analytic extension of~\eqref{simplified-G} to $(0,\pi)$.

\item Except for the degenerate case $b=0$, the formula~\eqref{simplified-G} is not valid for $\theta\in(\pi/2,\pi)$. It is an important property of $G$ that it is (when $b\not=0$) not symmetric with respect to $\theta=\pi/2$. This is the reason for the asymmetry of the resulting hSLEs mentioned in Section~\ref{intro:SLE}.

\item The function $G$ is real for $(\mu,\nu)$ in~\eqref{eq:range}. This is clear when $\mu\ge- (4-\kappa)^2/(16\kappa)$, since $a,b,c,d$ are all real in this case. When $\mu< -(4-\kappa)^2/(16\kappa)$, $c,d$ are still real and $a,b$ are complex conjugates. Making the series expansion~\eqref{eq:series} for the function~\eqref{simplified-G}, we see that all its terms have real coefficients, hence $G$ is real for $x\in(0,\pi/2)$. Since $G$ is analytic on $(0,\pi)$, it must be real on the whole interval.

\item In the limiting case $b=0$, we have $\Gamma(b)=\infty$, but $G$ is still well-defined as a limit and for all $\theta\in(0,\pi)$, we have 
\begin{align}\label{b=0}
G(\theta)=(\sin\theta)^{2d}.
 \end{align}
In this case, the SLE defined by~\eqref{driving-function} and~\eqref{radial eq} is just a radial SLE$_\kappa(\rho)$ for $\rho=2\kappa d$.
\end{itemize}

We have in fact chosen $G$ to be a solution of (\ref{diff-G}) (for some reasons that will be clear in Lemma \ref{lem:martingale}). See the following lemma.
\begin{lemma}\label{lem3.2}
The function $G$ satisfies the following differential equation
\begin{align}\label{diff-G}
e-\frac{\nu}{2\sin(\theta)^2} +\frac{G'(\theta)}{G(\theta)} \frac{\cot(\theta)}{2}+\frac{\kappa}{8} \frac{G''(\theta)}{G(\theta)} =0.
\end{align}
\end{lemma}
\begin{proof}
For $\theta\in(0,\pi/2)$, we can make the change of variable $z=(\sin\theta)^2$ and let $H(z)=G(\theta)$. Then (\ref{diff-G}) is equivalent to another differential equation for $H(z)$:
\begin{align}\label{diff-H}
e-\frac{\nu}{2z} +\frac{H'(z)}{H(z)}\left(\left(1+\frac{\kappa}{4}\right)-\left(1+\frac{\kappa}{2}\right)z \right) +\frac{H''(z)}{H(z)} \frac{\kappa}{2}z(1-z)=0.
\end{align}
The equation (\ref{diff-H}) is a modified hypergeometric differential equation and has 
\begin{align*}
H(z)=z^{{d}}{_2F_1}(a,b;c;z)
\end{align*}
as one of its two linearly independent  solutions.
This shows that \eqref{simplified-G} satisfies \eqref{diff-G} for $\theta\in(0,\pi/2)$.
Since \eqref{def1} is the analytic continuation of \eqref{simplified-G} and the coefficients of \eqref{diff-G} are clearly analytic on $(0,\pi)$, we get that \eqref{def1} is a solution of \eqref{diff-G}.
\end{proof}

Let us now derive some asymptotic behaviors of $G$ as $\theta$ tends to $0$ and $\pi$, which will be useful in the next section. We restrict ourselves to the case $b\not=0$ (since the case $b=0$ is simply given by~\eqref{b=0}).
\begin{itemize}
\item As  $\theta$ tends to $0$,  it is easy to see by (\ref{simplified-G}) that we have $G(\theta)\sim(\cot\theta)^{-2{d}}\sim \theta^{2{d}}$ and
\begin{align}\label{e1}
\frac{G'(\theta)}{G(\theta)}\underset{\theta\to 0}{=}\frac{2{d}}\theta + o(1/\theta).
\end{align}
\item As $\theta$ tends to $\pi$,
according to the identity (\ref{hyper-identity2}), 
the first and second terms of (\ref{def1}) are respectively equivalent to $C_1\cdot |\cot\theta|^{2c-2-2d}$ and $C_2\cdot |\cot\theta|^{2c-2-2d-1} (\cot\theta)$, where
\begin{align}
\notag
&C_1=\frac{\Gamma(c)\Gamma(1/2)}{\Gamma(c-a)\Gamma(c-b)}\cdot\frac{\Gamma(1/2)\Gamma(c-1)}{\Gamma(a)\Gamma(-1/2-a+c)},\\
\label{c2} &C_2=\frac{\Gamma(c)\Gamma(-1/2)}{\Gamma(a)\Gamma(b)}\cdot\frac{\Gamma(3/2)\Gamma(c-1)}{\Gamma(c-a)\Gamma(1/2+a)}.
\end{align}
Using the relation $a+b+1/2=c$ and the reflection identity $\Gamma(z)\Gamma(1-z)=\pi/\sin(\pi z)$, we can in fact deduce that $C_1=-C_2$. Note that $C_1=0$ if and only if $b=0$ (for $\nu, \mu$ in the range~\eqref{eq:range}), but we have ruled out this case.

Note that the same asymptotics hold as  $\theta\to0$: The first and second terms of (\ref{def1}) are also respectively equivalent to $C_1\cdot |\cot\theta|^{2c-2-2d}$ and $C_2\cdot |\cot\theta|^{2c-2-2d-1} (\cot\theta)$.
However, as $\theta\to 0$, we have $\cot(\theta)>0$, hence these two leading terms cancel out. This explains why $G(\theta)$ is equivalent to a higher order term $(\cot \theta)^{-2d}$.

However, as $\theta\to \pi$, since $\cot(\theta)<0$, these leading terms do not cancel out, yielding
\begin{align}\label{asym-G-pi}
G(\theta)\sim (C_1-C_2) \cdot |\pi-\theta|^{2d+2-2c}.
\end{align}
 and also
\begin{align}\label{e2}
\frac{G'(\theta)}{G(\theta)}\underset{\theta\to\pi}{=}\frac{2c-2-2d}{\pi-\theta} + O(1).
\end{align}
Note that we always have $2c-2-2d>0$.
\end{itemize}

Finally, we also prove the following lemma.
\begin{lemma}\label{lem:G>0}
For all $\theta\in(0, \pi)$, we have $G(\theta)>0$.
\end{lemma}
\begin{proof}
We define $K(\theta) = G(\theta) (\sin \theta)^{-2d}$. It suffices to prove that $K(\theta)>0$ for all $\theta\in(0, \pi)$.
For $\theta\in(0, \pi/2)$, by~\eqref{simplified-G}, we know that $K(\theta)= {_2F_1} (a,b;c; (\sin\theta)^2)$.
Note that for $\nu, \mu$ in the range~\eqref{eq:range}, all the coefficients of~\eqref{eq:series} are positive, hence $K(\theta)>0$ for $\theta\in(0, \pi/2)$. We also know from~\eqref{def1} that 
$$K(\pi/2)=G(\pi/2)=\frac{\Gamma( c)\Gamma(1/2)}{\Gamma( c- a)\Gamma( c- b)}>0.$$
From~\eqref{diff-G}, we deduce that $K$ satisfies the following differential equation
\begin{align}\label{diff-K}
e-d-\frac{\kappa d^2}{2} +\frac{\kappa d +1}{2} \cot\theta \frac{K'(\theta)}{K(\theta)} + \frac{\kappa}{8} \frac{K''(\theta)}{K(\theta)} =0.
\end{align}
Suppose the lemma is not true, so that there exists $\theta_0 \in (\pi/2, \pi)$ such that $K(\theta_0) =0$. Since $K$ is analytic, there exists $n\in\Nb$ such that $K(\theta)= (\theta -\theta_0)^n f(\theta)$, where $f$ is  an analytic function on $(0,\pi)$ with $f(\theta_0)\not=0$. Putting this expression back into~\eqref{diff-K}, in order to cancel out the poles at $\theta_0$, we deduce that we can only have $n=1$. This implies that $\theta_0$ is a simple root of $K$, so that there exists $\theta_1 \in(\pi/2, \pi)$ such that $K(\theta_1)<0$.

As $\theta$ tends to $\pi$, $G$ is equivalent to~\eqref{asym-G-pi}, which is in particular strictly positive, hence $K(\theta)$ is also strictly positive as $\theta$ tends to $\pi$. Since $K(\pi/2)>0$ and $K(\theta_1)<0$, we deduce that there exists $\theta_2\in(\pi/2, \pi)$ where $K$ attains its minimum, with $K(\theta_2)<0$, $K'(\theta_2)=0$ and $K''(\theta_2)\ge 0$.

The range~\eqref{eq:range} implies that $e-d- \kappa d^2 / 2\le 0$ (where the equality holds if and only if the equality in (ii) of~\eqref{eq:range} holds for $\mu$, i.e, $b=0$) and $ (\kappa d +1)/2 >0$. 
If $b=0$, then $G$ is given by~\eqref{b=0} which is strictly positive. Suppose that $b\not=0$ so that $e-d- \kappa d^2 / 2 > 0$.
At $\theta= \theta_2$, the left hand-side of~\eqref{diff-K} is strictly positive, which leads to a contradiction. This completes the proof.
\end{proof}

\subsection{Geometric properties of radial hSLE }\label{S:geom}
In this section, we will derive some basic geometric properties of the radial hSLE$_\kappa(\mu,\nu)$, which is defined to be the radial Loewner evolution encoded by~\eqref{radial eq}, with driving function $W_t$ determined by $G$ and~\eqref{driving-function}.

Let $\gamma$ be a radial hSLE$_\kappa(\mu, \nu)$ starting at $1$ with a marked point $x_0=e^{i 2\theta_0}$ where $\theta_0\in(0,\pi)$. 
Define the stopping time 
\begin{align}\label{stoping_time}
T:=\inf\{t>0: \theta_t=0 \text{ or } \theta_t=\pi\}.
\end{align}
Note that if $T<\infty$, then $T$ is also the first time that $\gamma$ disconnects $x_0$ from $0$ in $\Ub$. Let us first prove the following lemma.

\begin{lemma}\label{lem:gamma_abs_cont}
For all $t<T$, the process $\gamma([0,t])$ is a.s.\ a simple curve which stays in $\Ub$ (except its starting point $\gamma(0)$).
\end{lemma}

\begin{proof}
Note that on $(0,\pi)$, $G$ is $C^\infty$ and $G(\theta)>0$ by Lemma~\ref{lem:G>0}, hence the drift term $G'(\theta_t)/G(\theta_t)$ in ~\eqref{driving-function} is bounded on any compact subset of $(0,\pi)$.
By Girsanov theorem, $W_t$ is absolutely continuous w.r.t.\ $\sqrt{\kappa} B_t$, when $\theta_t$ belongs to any given compact subset of $(0, \pi)$. 
More precisely, for $\eps\in(0,\pi/2)$, define the stopping time 
\begin{align}\label{stopping_time}
T_\eps:=\inf\{t>0: \theta_t\le \eps \text{ or } \theta_t>\pi-\eps\}.
\end{align}
Then $\gamma|_{t\in[0, T_\eps]}$ is absolutely continuous w.r.t.\ an ordinary radial SLE$_\kappa$. In particular, it is a.s.\ a simple curve which does not hit the boundary $\partial \Ub$.
Since this is true for all $\eps$, it follows that  for any $t<T$, $\gamma([0,t])$ is a.s.\  simple which does not hit the boundary. 
\end{proof}

Now, we analyse the behavior of $\gamma$ as $t$ tends to $T$.  We aim to prove the following proposition.

\begin{proposition}\label{prop:hsle_geom}
Radial hSLE$_\kappa(\mu,\nu)$ has the following  properties (see Figure \ref{fig:hsle1}):
\begin{enumerate}[(i)]
\item If $b=0$, then it is a radial SLE$_\kappa(\rho)$ with $\rho=2\kappa d$.
In particular, we have $T=\infty$ and $\gamma(\infty)=0$ a.s. See Figure~\ref{fig:hsle1}(i).

\item If $b\not=0$ and $\nu\ge 1/2-\kappa/16$, then $T<\infty$ and $\gamma(T)=x_0$ a.s. Moreover, $\gamma([0,T])$ a.s.\ leaves the origin on its left. See Figure~\ref{fig:hsle1}(ii).

\item If $b\not=0$ and $\nu< 1/2-\kappa/16$, then $T<\infty$ and $\gamma(T) \in \ell_t$ a.s., where $\ell_t$ is  the counterclockwise part  of the boundary of $\Ub\setminus \gamma[0,t]$ from $x_0$ to $\gamma_t$. Moreover, $\gamma([0,T])$ a.s.\ leaves the origin on its left. See Figure~\ref{fig:hsle1}(iii).
\end{enumerate}
\end{proposition}

\begin{proof}
In case (i) when $b=0$, by~\eqref{b=0}, we have $G'(\theta)/ G(\theta) =2d \cot(\theta)$. Putting it back into~\eqref{driving-function}, we get the same equation as that of a radial SLE$_\kappa(2\kappa d)$.

In the rest of the proof, we consider $b\not=0$.
Note that $\theta_t$ is solution to the stochastic differential equation
\begin{align}\label{eq:theta}
d\theta_t= \frac{\sqrt{\kappa}}{2} dB_t+\frac{\kappa}{4}\frac{G'(\theta_t)}{G(\theta_t)}dt +\frac12\cot\theta_t dt.
\end{align}
The drift term in \eqref{eq:theta} is continuous in $\theta_t$ for $\theta_t\in (0,\pi)$ and tends to $\infty$ when $\theta_t$ approaches $0$ and $\pi$ with respective speeds
$$(\frac{\kappa}{2}d+1/2)\frac{1}{\theta_t}+ O(1) \text{ as } \theta_t \to 0 \quad \text{and} \quad  \!\left(-\frac{\kappa}{4}(2d+2-2c)-\frac12 \right)\frac{1}{\pi-\theta_t}+O(1) \text{ as } \theta_t \to \pi$$
due to (\ref{e1}) and (\ref{e2}).
One can then make a Girsanov transformation and show that:
\begin{itemize}
\item When $\theta_t$ is in a neighborhood of $0$, by (\ref{e1}), it is absolutely continuous w.r.t.\ the solution of the following SDE:
\begin{align}\label{SDE1}
d\theta_t= \frac{\sqrt{\kappa}}{2} dB_t+(\frac{\kappa}{2}d+1/2)\frac{1}{\theta_t} dt,
\end{align}
which is a Bessel process with dimension $2+4q_2(\kappa,\nu)>2$.  Hence $\theta_t$ will a.s.\ not hit $0$.
This implies that $\gamma|_{t\in [0,T]}$ will a.s.\ not not disconnect $x_0$ from $0$ in a way that leaves $0$ on its right. 

\item As we argued in the proof of Lemma~\ref{lem:gamma_abs_cont},  on any compact subinterval of $(0, \pi)$, $G'(\theta)/G(\theta)$ is bounded, and so is $\cot(\theta)$. 
Consequently, by~\eqref{eq:theta}, for any $\eps>0$, before the stopping time $T_\eps$ defined in~\eqref{stopping_time}, $\theta_t$ is absolutely continuous w.r.t.\ $(\sqrt{\kappa}/2) B_t$.
Since the Brownian motion started in $(\eps, \pi-\eps)$ a.s.\ hits $\eps$ or $\pi-\eps$ at a finite time, so does $\theta_t$. We have argued in the previous bullet point that $\theta_t$ a.s.\ does not hit $0$, hence if $\eps$ is sufficiently small, each time $\theta_t$ hits $\eps$, it will a.s.\ come back to $2\eps$. Then starting from $2\eps$, $\theta_t$ has a positive probability of hitting $\pi-\eps$ before $\eps$. This implies that $\theta_t$ will end up hitting $\pi-\eps$ at a finite time. This is true for all $\eps>0$.
When $\theta_t$ is in a neighborhood of $\pi$, by (\ref{e2}), the process $\pi-\theta_t$ is absolutely continuous w.r.t.\ the solution of the following SDE:
\begin{align}\label{SDE2}
d \omega_t = \frac{\sqrt{\kappa}}{2} dB_t+ \!\left( \frac{\kappa}{4}(2d+2-2c)+1/2\right) \frac{1}{\omega_t} dt,
\end{align}
which is a Bessel process with dimension $2-4q_2(\kappa,\nu)<2$.  Hence $\theta_t$  will  a.s.\ hit $\pi$ at a finite time.
This implies that $\gamma|_{t\in [0,T]}$ will  a.s.\ disconnect $x_0$ from the origin in a way that leaves $0$ on its left. 
\end{itemize}

The above arguments show that as long as $b\not=0$, $\gamma|_{t\in [0,T]}$ will a.s.\ disconnect $x_0$ from the origin in a way that leaves $0$ on its left. 
At the disconnection time, $\gamma$ can either hit exactly $x_0$, or hit some other point on $\ell_t$, always leaving the origin on its left.
To see which case we are in, one needs to do a coordinate change (of the type \cite{MR2188260}).

\vspace{1mm}
\begin{figure}[h!]
\centering
\includegraphics[width=0.8\textwidth]{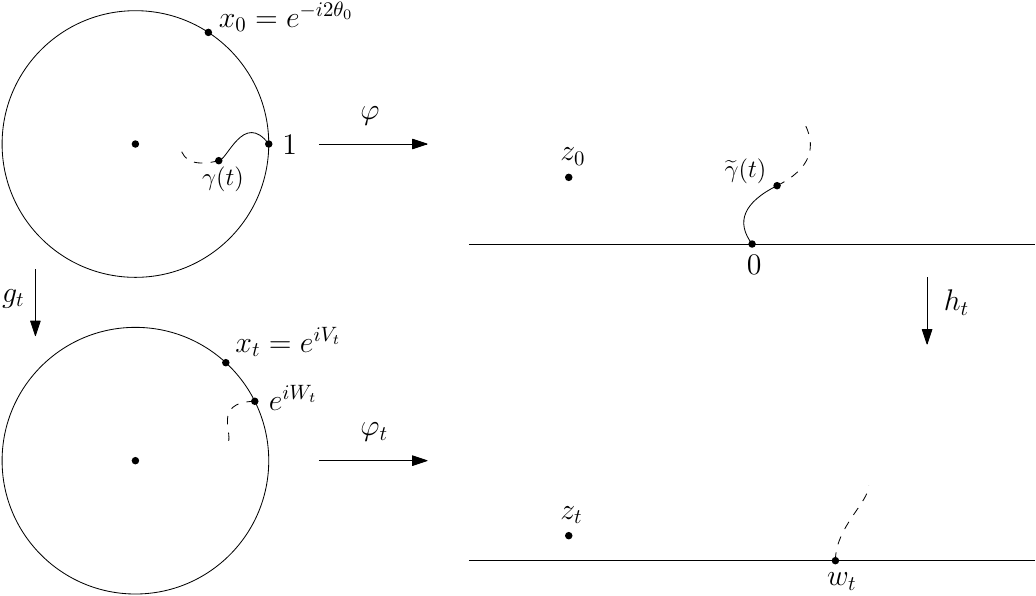}
\caption{The change of coordinate for the radial hSLE}
\label{fig:coordinate}
\end{figure}

We want to map $\gamma$ to the upper half-plane by sending $1, x_0$ to $0, \infty$, so that we can view the image of $\gamma$ as a chordal SLE from $0$ to $\infty$. The event that $\gamma(T)=x_0$ will then be the same as the event that the image of $\gamma$ goes to $\infty$ without hitting the boundary $\Rb$. 

More precisely (see Figure \ref{fig:coordinate}), let $\varphi$ be the M\"obius map  from $\Ub$ onto $\Hb$ that sends $1, x_0$ to $0, \infty$ and given by
\begin{align*}
\varphi(x)= C \frac{\psi(x)}{\psi(x)+\tan(\theta_0)},
\end{align*}
where  $C\in\Rb$ and $\psi(x)=i(1-x)/(1+x)$, so that $\psi^{-1}(z)=(i-z)/(i+z)$.
This implies
\begin{align}\label{eq:phi-1}
\varphi^{-1}(z)=\frac{i(C-z)-z\tan(\theta_0)}{i(C-z)+z\tan(\theta_0)}=e^{-2i\theta_0} \!\left(1-\frac{iC \sin(2\theta_0)}{z}+O(1/z^2) \right) \text{ as } z\to \infty.
\end{align}
For all $t<T$, let $x_t=g_t(x_0)=e^{iV_t}$. Let $\wt \gamma$ be the image of $\gamma$, parametrized in a way that $\wt \gamma(t)= \varphi(\gamma(t))$.
Let $h_t$ be the conformal map from $\Hb\setminus \wt\gamma([0,t])$ onto $\Hb$, normalized at infinity in a way that there exist $s(t)\in \Rb$, so that
\begin{align}\label{eq:h_t}
h_{t}(z)= z+{2 s(t)}/{z} +O\left({1}/{z^2}\right) \, \text{ as } \, z\to\infty.
\end{align}
Let $\varphi_t$ be the M\"obius map from $\Ub$ onto $\Hb$ given by $h_t \circ \varphi \circ g_t^{-1}$.
Let $w_t=\varphi_t(e^{iW_t})$ and $z_t=\varphi_t(0)$.
Note that since  $\varphi_t$ sends $x_t$ to $\infty$, there exists $C_t\in\Rb$ so that
\begin{align}\label{eq:varphi_t}
\varphi_t(x)=w_t+ C_t \frac{\psi(e^{-iW_t}x)}{\psi(e^{-iW_t} x)+\tan(\theta_t)}=w_t+ \frac{iC_t(e^{iW_t}-x)}{(\tan(\theta_t)-i)(x-x_t)}.
\end{align}
Using~\eqref{eq:phi-1} and developing $g_t \circ \varphi^{-1}$ at $z=\infty$, we get
\begin{align}\label{d1}
g_t(\varphi^{-1}(z))-x_t=-  g_t'(x_0)  iC \sin(2\theta_0) e^{-2i\theta_0} \frac{1}{z} +O(1/z^2).
\end{align}
Combining~\eqref{eq:varphi_t} and~\eqref{d1}, we can then develop $h_t=\varphi_t\circ g_t \circ \varphi^{-1}$ at $z=\infty$ and get its leading term 
\begin{align}\label{eq:coef_z}
-\frac{iC_t(e^{iW_t}-x_t)}{(\tan(\theta_t)-i)e^{-2i\theta_0} g_t'(x_0) iC\sin(2\theta_0)} z.
\end{align}
By~\eqref{eq:h_t}, the coefficient of $z$ given by~\eqref{eq:coef_z} should be equal to $1$, yielding
\begin{align}\label{ctc}
C_t/C=\frac{(\tan(\theta_t)-i) g_t'(x_0) \sin(2\theta_0)}{(x_t-e^{iW_t})e^{2i\theta_0}}=e^{-iV_t} g_t'(x_0) e^{-2i\theta_0}  \sin(2\theta_0)/ \sin(2\theta_t).
\end{align}
Now we want to inspect the quantity $X_t:=|z_t-w_t|$. The chordal SLE $\wt\gamma$ goes to infinity if and only if $X_t$ goes to infinity.
Note that by~\eqref{eq:varphi_t}, we have
\begin{align*}
z_t-w_t=C_t \cos(\theta_t) e^{i\theta_t}.
\end{align*}
By~\eqref{ctc}, this implies that for some $C'\in\Rb$, we have
\begin{align}\label{xt}
X_t=C_t\cos(\theta_t)=C' e^{-iV_t} g_t'(x_0) /\sin(\theta_t).
\end{align}
Applying It\^o calculus to~\eqref{xt}, we get that
\begin{align}\label{eq:X_t}
d\log X_t= -\frac{1}{2(\sin\theta_t)^2}-\frac{\cot(\theta_t)}{2} \!\left(\sqrt{\kappa} dB_t+\frac{\kappa}{2}\frac{G'(\theta_t)}{G(\theta_t)} +\cot(\theta_t) \right)dt+(1+(\cot\theta_t)^2)\frac{\kappa}{8} dt.
\end{align}
Note that as $t\to T$, we have $\theta_t\to\pi$, so $\cot\theta_t\sim -\sin\theta_t\to 0$. Making the change of time $r(t)=\int_0^t (\cot\theta_s)^2 ds$ for $t\in[0,T]$ so that $r\in [0,\infty)$, and letting $\wt X_{r(t)}=X_{t}$, we get that as $t\to T$ (hence as $r\to \infty$), we have
\begin{align*}
d\log \wt X_r=\frac{\sqrt{\kappa}}{2} dB_r + \!\left( \frac{\kappa}{8}-1-\frac{\kappa}{4} (2d+2-2c) + o(1) \right) dr,
\end{align*}
where the term $o(1)$ tends to $0$ as $\theta_t\to \pi$ (hence as $r\to \infty$).
The process $\log X_t$ converges to $-\infty$ as $t\to T$ if and only if $\log \wt X_r$ converges to $-\infty$ as $r\to\infty$. This happens if and only if
\begin{align*}
\frac{\kappa}{8}-1-\frac{\kappa}{4} (2d+2-2c)=\frac{\kappa}{2} q_2(\kappa,\nu) -\frac12 <0.
\end{align*}
The inequality above is also equivalent to $\nu < \frac12-\kappa/16$.
Therefore, if $\nu \ge \frac12-\kappa/16$, then the chordal curve $\wt\gamma$ goes to $\infty$, corresponding to the event that $\gamma$ hits $x_0$, as in case (ii) of Proposition~\ref{prop:hsle_geom}. Otherwise if $\nu < \frac12-\kappa/16$, then $\wt\gamma$ hits $\Rb^-$ swallowing $z$, corresponding to the event that $\gamma(T) \in \ell_t$, as in case (iii) of Proposition~\ref{prop:hsle_geom}.
\end{proof}

\section{Construction of general radial restriction measures}\label{Sec:cons}
In this section, we will construct general radial restriction measures using the radial hSLE processes, and consequently prove Theorem~\ref{intro-thm:restriction}.
The strategy of this section is similar to that of \cite{MR1992830} (also see \cite[Section 5]{MR2060031}). Some extra care is needed in the analysis of the limiting behavior as the hSLEs reach their ends.

Throughout, we fix $\kappa\in(0,4]$ and $\mu, \nu$ in the range~\eqref{eq:range}. Let $a,b,c,d,e$ be given by~\eqref{eq:parameters}.
We will first construct a random set $K$ in Section~\ref{SS restriction}, then prove in Section~\ref{SS martingale} that it indeed satisfies radial $\kappa$-restriction and determine its parameters.

\subsection{Method of construction}\label{SS restriction}

Let us now explicitly construct the random set $K$. Let $\gamma$ be a radial hSLE$_\kappa(\mu,\nu)$ from $1$ aiming at the origin with a marked point at $e^{i0^-}$ (directly below $1$). Let $T$ be the stopping time defined by~\eqref{stoping_time}. The definition of $K$ depends on the parameters $\mu,\nu$ as follows, see Figure~\ref{hsle2}.

\vspace{1mm}
\begin{figure}[h!]
\centering
\includegraphics[width=\textwidth]{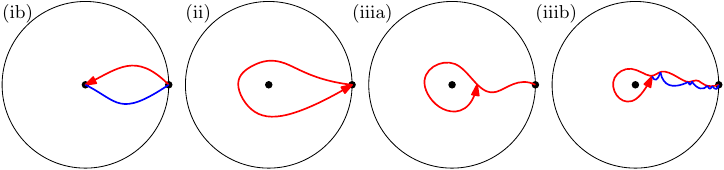}
\caption{Illustration of the construction of $K$ in different cases. We first run the red curves which are hSLE. Then in some cases, we run a second blue curve which is distributed as SLE$_\kappa(\rho)$.}
\label{hsle2}
\end{figure}
\vspace{-1mm}

\begin{enumerate}[(i)]
\item If $b=0$, then $\gamma$ is an SLE$_\kappa(\rho)$ curve where $\rho=2\kappa d$, according to (i) in Proposition~\ref{prop:hsle_geom}.
\vspace{-1mm}
\begin{itemize}
\item[(ia)] If $\nu=0$, then $\rho=0$, hence $\gamma$ is a radial SLE$_\kappa$. Let $K$ be $\gamma([0,\infty])$ which is a simple curve from $1$ to $0$.
\item[(ib)] If $\nu>0$, then in the domain $\Ub\setminus \gamma$, we grow a chordal SLE$_\kappa(\rho-2)$ curve $\gamma'$ from $1$ to $0$ with a marked point immediately to its right. Let $K$ be the compact set enclosed by $\gamma$ and $\gamma'$.
See Figure~\ref{hsle2}(ib).
\end{itemize}

\item If $b\not=0$ and $\nu\ge\frac12-\kappa/16$, then $\gamma$ will a.s.\ make a simple counterclockwise loop around the origin before returning to $1$ at time $T$, according to  (ii) in Proposition~\ref{prop:hsle_geom}. Let $K$ be the compact set enclosed by $\gamma([0,T])$. See Figure \ref{hsle2}(ii). 

\item  If $b\not=0$ and $\nu\in\!\left[0,\frac12-\kappa/16\right)$, then $\gamma$ will a.s.\ make a counterclockwise loop around the origin, before  intersecting its own left boundary at time $T$, according to  (iii) in Proposition~\ref{prop:hsle_geom}. 
\vspace{-1mm}
\begin{itemize}
\item[(iiia)] If $\nu=0$, then let $K$ be the compact set enclosed by $\gamma([0,T])$. See Figure~\ref{hsle2}(iiia).

\item[(iiib)] If $\nu>0$, then in the connected component of $\Ub\setminus \gamma([0,T])$ which does not contain the origin, we grow a SLE$_\kappa(\rho)$ curve $\gamma'$ from $1$ to $\gamma(T)$ with a marked point immediately to its right. Then let $K$ be the compact set enclosed by $\gamma$ and $\gamma'$. See Figure~\ref{hsle2}(iiib).
\end{itemize}
\end{enumerate}

\subsection{General restriction property} \label{SS martingale}
The goal of this section is to prove the following proposition.

\begin{proposition}\label{prop-restriction}
The set $K$ constructed in Section~\ref{SS restriction} satisfies radial $\kappa$-restriction property with exponents
\vspace{-3mm}
\begin{align}\label{alpha-beta}
\alpha=2 \mu,\quad \beta=\frac{1}\kappa+\nu+2 q_2(\kappa,\nu).
\end{align}
\end{proposition}
We first remark that Proposition~\ref{prop-restriction} does imply Theorem~\ref{intro-thm:restriction}:
 The method in Section~\ref{SS restriction} constructs all $\kappa$-restriction measures with exponent $(\alpha,\beta)$ for $\kappa\in(0,4]$   and $(\alpha,\beta)$ in the range
$$\alpha\le \eta_\kappa(\beta), \quad \beta\ge (6-\kappa)/(2\kappa).$$
More precisely, the different cases of Section~\ref{SS restriction} correspond to the following ranges of parameters:
\begin{itemize}
\item[(ia)]  corresponds to $\alpha=\eta_\kappa(\beta), \beta= (6-\kappa)/(2\kappa)$; \vspace{-1mm}
\item[(ib)] corresponds to $\alpha=\eta_\kappa(\beta), \beta\ge (6-\kappa)/(2\kappa)$; \vspace{-1mm} 
\item[(ii)] corresponds to $\alpha< \eta_\kappa(\beta), \beta\ge (12-\kappa)(\kappa+4)/(16\kappa)$; \vspace{-1mm}
\item[(iiia)] corresponds to $\alpha< \eta_\kappa(\beta), \beta=(6-\kappa)/(2\kappa)$;  \vspace{-1mm}
\item[(iiib)] corresponds to $\alpha< \eta_\kappa(\beta),  \beta\in ((6-\kappa)/(2\kappa), (12-\kappa)(\kappa+4)/(16\kappa))$.
\end{itemize}
 The geometric properties of radial hSLEs given by Proposition~\ref{prop:hsle_geom} then imply the geometric properties of radial $\kappa$-restriction measures, as described in Theorem~\ref{intro-thm:restriction} (also see Figure~\ref{fig:restriction}).

To prove Proposition~\ref{prop-restriction}, we will  rely on an appropriate local martingale, given in Lemma~\ref{lem:martingale}.
Let us first define some quantities. See Figure~\ref{commutation}.
\begin{figure}[h!]
\centering
\includegraphics[width=0.6\textwidth]{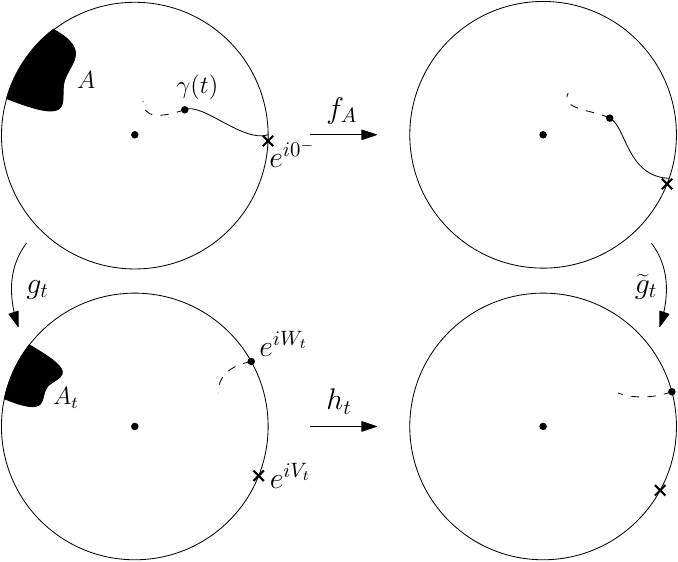}
\caption{The commutation diagram for the conformal maps $h_t$, $g_t$, $\wt g_t$, $f_A$.}
\label{commutation}
\end{figure}
Let $\Fc_t$ be the filtration of the Brownian motion used to generate $\gamma$. For all $A\in\Qc$, let $\tau$ be the first time that $\gamma$ intersects $A$. Recall $T$ is the stopping time defined by~\eqref{stoping_time}.
For all $t<T\wedge \tau$, let $A_t=g_t(A)$ and $h_t=f_{A_t}$.
Let $\nu_t=\frac12\arg (h_t(e^{iW_t})-h_t(e^{iV_t})).$ 
We also recall the Schwarzian derivative
\begin{align}\label{eq:Sch}
S f(z)=\frac{f'''(z)}{f'(z)} -\frac{3f''(z)^2}{2 f'(z)^2}.
\end{align}

\begin{lemma}\label{lem:martingale}
The following is a local martingale for the filtration $(\Fc_t)_{0\le t<T\wedge \tau}$:
\begin{align*}
M_t=&|h_t'(0)|^{2\mu} \, |h_t'(e^{iW_t})|^{(6-\kappa)/(2\kappa)} \, |h_t'(e^{iV_t})|^\nu \, \frac{G(\nu_t)}{G(\theta_t)}\, \exp\left( \int_0^t -\frac{c(\kappa)}{6} |Sh_s(W_s)| ds \right).
\end{align*}
\end{lemma}
\begin{remark}\label{remark}
We remark that, the form of the local martingale $(M_t)$ above can be guessed from the form of the driving function~\eqref{driving-function}. 
One possible point of view is to see $G$ as the ``partition function'' of the hSLE (see \cite{MR2518970,MR2571956}) and then guess the form of this local martingale in analogy to the restriction martingales for SLE$_\kappa(\rho)$s.
However, such arguments can not be made as a proof of this lemma. 
We will actually prove it by performing It\^o calculus, and we postpone it to Appendix \ref{A3}. 
In particular, as we will show in Appendix \ref{A3}, $(M_t)$ is a local martingale only if  $G$ satisfies~\eqref{diff-G}.
\end{remark}

For the moment, $M_t$ is only defined for $t<T\wedge \tau$. We will restrict ourselves on the event $\{T<\tau\}$ or
equivalently $\{\gamma\cap A=\emptyset\}$ (here and in the sequel, we sometimes denote the set $\gamma([0,T])$ by $\gamma$) and will  define $M_{T}$ as the limit of $M_t$ as $t\to T$ (note that if $T>\tau$, then $M_\tau=0$ a.s.,  although we do not use this fact). 
 Let us now analyse the behavior of $M_t$ as $t\to T$.

\begin{lemma}\label{lem:limit}
We restrict ourselves on the event $\gamma \cap A=\emptyset$. For all the cases (ia), (ib), (ii), (iiia) and (iiib), as $t\to T$, the three terms  $|h_t'(0)|$,  $ |h_t'(e^{iW_t})|$ and $G(\nu_t)/G(\theta_t)$  tend to $1$.
For the case (ii), we have in addition that $ |h_t'(e^{iV_t})|$ tends to $1$.
\end{lemma}

\begin{proof}
We will prove the lemma for the cases (ii) and (iiia), (iiib). The proof in the cases  (ia) and (ib) is slightly different, but is in fact simpler and follows from the same type of reasoning, hence we leave it to the reader.
We illustrate the case (iiib) in Figure~\ref{fig:limit}. 

Suppose that we are in the cases (ii), (iiia) or (iiib).
Since $\gamma([0,T])$ forms a closed loop around the origin, as $t$ tends to $T$, the harmonic measure seen from the origin in the domain $\Ub\setminus \gamma([0,t])$ of the counterclockwise part of boundary from $e^{i0^-}$ to $\gamma(t)$ tends to $1$. 
This implies that after conformally mapping $\Ub\setminus \gamma([0,t])$ to $\Ub$ by $g_t$,  the point $e^{iV_t}$ is counterclockwisely very close to $e^{iW_t}$.
Moreover, $A_t$ is attached to the counterclockwise arc from $e^{iW_t}$ to $e^{iV_t}$. The harmonic measure of $A_t$ seen from the origin is also small, because a Brownian motion started at the origin has very small probability of hitting $A_t$ before $\gamma([0,t])$ as $t\to T$ (since it has to first exit the ``quasi-loop'' formed by $\gamma([0,t])$ without hitting it). This already implies that  $|h_t'(0)|$ tends to $1$.

Let $a_t$ be the point in $A_t\cap \partial \Ub$ which is the closest to $e^{iW_t}$. The harmonic measure seen from the origin of the counterclockwise arc from $e^{iW_t}$ to $a_t$ is much larger than the harmonic measure of $A_t$. This is because in $\Ub\setminus \gamma([0,t])$, if we condition a Brownian motion started from the origin to stop at the clockwise part of boundary from $e^{i0^-}$ to $\gamma(t)$, then with conditional probability tending to $1$, it is going to stop in a neighborhood of the tip $\gamma(t)$, rather than hitting $A_t$. This is because in order to hit that part of boundary, the Brownian motion has to first exit the ``quasi-loop'', hence get very close to the tip $\gamma(t)$. From there, it has macroscopic distance to $A_t$, hence is much more likely to hit somewhere near $\gamma(t)$ before $A_t$.
This proves that $ |h_t'(e^{iW_t})|$ tends to $1$.
Finally, by~\eqref{asym-G-pi}, we know that $G(\nu_t)/G(\theta_t)$ is asymptotical to $( H_t/|W_t -V_t| )^{2d+2-2c}$, where  $H_t$ is the harmonic measure seen from the origin in $\Ub$ of the counterclockwise arc from $h_t(e^{iW_t})$ to $h_t(e^{iV_t})$. Since the harmonic measure of $A_t$ in $\Ub$ seen from $0$ is much smaller than $|W_t-V_t|$, we have that $( H_t/|W_t -V_t| )$ tends to $1$. Therefore $G(\nu_t)/G(\theta_t)$ also tends to $1$ as $t\to T$.

\begin{figure}[h]
\centering
\includegraphics[width=0.7\textwidth]{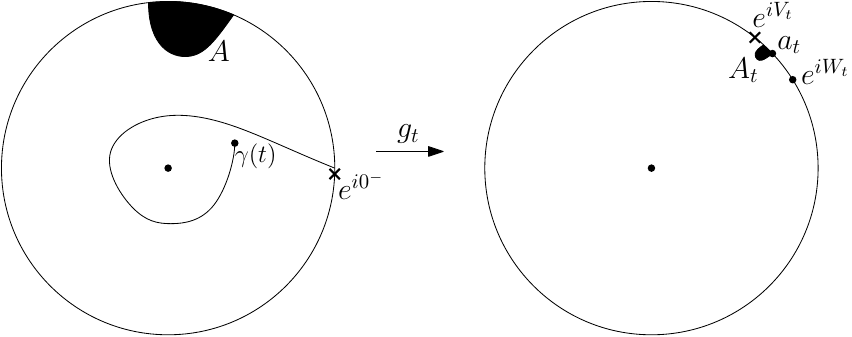}
\caption{Limiting behavior as $t\to T$ in case (iiib): The counterclockwise arc from $e^{iW_t}$ to $e^{iV_t}$ gets very small, and $A_t$ is attached to that arc. The size of $A_t$ is again much smaller than $|V_t-W_t|$. Moreover, $A_t$ is much closer to $e^{iV_t}$ than to $e^{iW_t}$. This explains why $ |h_t'(e^{iW_t})|$ tends to $1$, but $ |h_t'(e^{iV_t})|$ does not.}
\label{fig:limit}
\end{figure}

Now, suppose that we are in case (ii), and let us show that $ |h_t'(e^{iV_t})|$ tends to $1$ as $t\to T$. 
Let $b_t$ be the point in $A_t\cap \partial \Ub$ which is the closest to $e^{iV_t}$.
Then, seen from $0$, the harmonic measure of the counterclockwise arc from $b_t$ to $e^{iV_t}$ is much bigger than the harmonic measure of $A_t$. Indeed, in $\Ub\setminus \gamma([0,t])$, if we condition a Brownian motion started from the origin to stop at the clockwise part of boundary from $e^{i0^-}$ to $\gamma(t)$, then for any fixed $\eps>0$ (we would choose $\eps$ small enough so that $A$ is disjoint from the clockwise arc from $e^{i0^-}$ to $e^{-i\eps}$), with conditional probability bounded from below as $t\to T$, it will stop on the clockwise arc from $e^{i0^-}$ to $e^{-i\eps}$ (Note that $\gamma(t)$ tends to $e^{i0^+}$ in case (ii). In other cases, this statement is not true).
However, the conditional probability that such a Brownian motion hits $A$ tends to $0$. Applying the conformal map $g_t$, it then follows that in the image domain $\Ub$, seen from $0$, the harmonic measure of $A_t$ is much smaller than that of the counterclockwise arc from $b_t$ to $e^{iV_t}$. This then implies that $ |h_t'(e^{iV_t})|$ tends to $1$.
\end{proof}

It now only remains to analyse the limit of $ |h_t'(e^{iV_t})|$ for the cases (ib) and (iiib) (for the cases (ia) and (iiia), since $\nu=0$, the term $ |h_t'(e^{iV_t})|$ does not exist in the local martingale $M_t$).  Let $D_\gamma$ be the connected component of $\Ub\setminus \gamma([0,T])$ which is connected to $\partial \Ub$.
Let $\wh f$ be some conformal map from $D_\gamma$ onto $\Ub$ that sends $\gamma(T)$ to $-1$ and $1$ to $1$. Let $\wh h_A$ be some conformal map from $\Ub\setminus \wh f(A)$ onto $\Ub$ that leaves $-1,1$ fixed. There is one degree of freedom in the choice of $\wh f$ or $\wh h_A$ (and we will fix them later), but the quantity $|\wh h_A'(-1) \wh h_A'(1)|$ doesn't depend on the choice of $\wh f$ or $\wh h_A$, and we denote it by $C$.

The quantity $C$ has the following meaning: Note that, conditionally on $\gamma$, in  the domain $D_\gamma$, $\gamma'$ is a SLE$_\kappa(\rho-2)$, hence satisfies (one-sided) chordal $\kappa$-restriction with parameter $\nu$ (see Definition~\ref{def:gen_chordal} and remark below): 
Let $\Pb^\gamma$ be the conditional law of $\gamma'$.
Let $f_{\gamma, A}$ be a conformal map from $D_\gamma$ onto  $D_\gamma\setminus A$ that leaves $\gamma(T)$ and $1$ fixed. There is one degree of freedom in the choice of $f_{\gamma, A}$, but the (conditional) law of $f_{\gamma, A}(\gamma')$ is independent of the choice, and we denote it by $\Pb^\gamma_A$.
Then we have
\begin{align}\label{eq:C}
\frac{d\Pb^\gamma(\gamma')}{d \Pb^\gamma_A(\gamma')} \mathbf{1}_{\gamma'\cap A=\emptyset} =\mathbf{1}_{\gamma'\cap A=\emptyset}\, C^\nu \exp\!\left(-c(\kappa) m_{D_\gamma}(\gamma', A) \right).
\end{align}
We now state the following lemma.

\begin{lemma}\label{lem:limit2}
On the event $\gamma\cap A=\emptyset$, in the cases (ib) and (iiib),  as $t\to T$, $ |h_t'(e^{iV_t})|$ tends to $C$.
\end{lemma}
\begin{proof}
We will prove the lemma for the case (iiib). The case (ib) is easier and can be proven with similar ideas.

Fix $x_1, x_2\in\partial \Ub$ distinct and $t\in(0,T)$. Let $\wh f_t$ be some conformal map from $\Ub\setminus \gamma([0,t])$ onto $\Ub$ that sends $\gamma(t)$ to $x_1$ and $1$ to $x_2$.
Let $\wh h_t$ be some conformal map from $\Ub\setminus \wh f_t(A)$ onto $\Ub$ that leaves $x_1,x_2$ fixed. 
There is one degree of freedom in the choice of $\wh f_t$ or $\wh h_t$ (and we will fix them later), but the quantity $|\wh h_t'(-1) \wh h_t'(1)|$ doesn't depend on the choices of $\wh f_t, \wh h_t$ or the points $x_1, x_2$, and we denote it by $C_t$.

Here is one possible set of choices: We fix $x_1=-1$ and $x_2=1$. We choose $\wh f_t$ with the additional condition that $\wh f_t'(-1)=1$. Then, as $t\to T$, $\wh f_t$ converges to a conformal map $\wh f$ from $D_\gamma$ onto $\Ub$ that sends $\gamma(T)$ to $-1$ and $1$ to $1$ and such that $\wh f'(-1)=1$. This implies that $C_t$ tends to $C$ as $t\to T$.

Here is another possible set of choices: For each $t\in (0,T)$, we fix $x_1=e^{iW_t}$, $x_2=e^{i V_t}$ and $\wh f_t =g_t$.
Let $\wh h_t=s_t\circ h_t$, where $s_t$ is some conformal map from $\Ub$ onto itself that sends $h_t(e^{iW_t})$, $h_t(e^{iV_t})$ to $e^{iW_t}, e^{iV_t}$. The quantity $|s_t'(h_t(e^{iW_t})) s_t'(h_t(e^{iV_t}))|$ doesn't depend on the choice of $s_t$. As we have explained in the proof of Lemma~\ref{lem:limit} (also see Figure~\ref{fig:limit}), the points $e^{iW_t}$ and $e^{iV_t}$ tend to be very close and so do the points $h_t(e^{iW_t})$ and $h_t(e^{iV_t})$.
Therefore $|s_t'(h_t(e^{iW_t})) s_t'(h_t(e^{iV_t}))|$ is asymptotical to  $H_t/|W_t -V_t| $, where  $H_t$ is the harmonic measure seen from the origin in $\Ub$ of the counterclockwise arc from $h_t(e^{iW_t})$ to $h_t(e^{iV_t})$. This quantity tends to $1$ by the proof of Lemma~\ref{lem:limit}. This implies that $C_t$ is asymptotical to $|h_t'(e^{iW_t}) h_t'(e^{iV_t})|$. Since $|h_t'(e^{iW_t})|$ tends to $1$, we actually have that  $C_t$ is asymptotical to $|h_t'(e^{iV_t})|$. Since $C_t$ tends to $C$, this implies the lemma. 
\end{proof}

\begin{lemma}\label{lem:limitMT}
On the event $\gamma\cap A=\emptyset$, as $t\to T$, $M_t$ tends to a limit which we denote by $M_{T}$. Moreover,
\begin{itemize}
\item For the cases (ia), (ii), (iiia), we have $M_{T} \mathbf{1}_{\gamma\cap A=\emptyset}=\mathbf{1}_{\gamma\cap A=\emptyset}\, \exp \!\left(-  c(\kappa) m_\Ub(\gamma, A) \right).$
\item For the cases (ib), (iiib), we have $M_{T} \mathbf{1}_{\gamma\cap A=\emptyset}=\mathbf{1}_{\gamma \cap A=\emptyset}\, \exp \!\left( - c(\kappa) m_\Ub(K, A) \right)\, C^\nu.$
\end{itemize}
\end{lemma}
\begin{proof}
We restrict ourselves on the event $\gamma \cap A=\emptyset$. For all $t\in[0,T]$, we have
\begin{align}\label{eq:loop_mass_t}
m_\Ub(\gamma([0,t]), A)= \int_0^t \frac{1}{6} |Sh_s(W_s)| ds.
\end{align}
One can see \cite[equation (23)]{MR2045953} for a chordal version of this formula in the upper half plane $\Hb$. The above formula is a radial version, which can be obtained from the chordal version by a coordinate change.
Consequently, we have
\begin{align*}
\exp\left( \int_0^T -\frac{c(\kappa)}{6} |Sh_s(W_s)| ds \right) = \exp \!\left( - c(\kappa) m_\Ub(\gamma, A) \right).
\end{align*}
Combining with Lemma~\ref{lem:limit} and~\ref{lem:limit2}, we complete the proof.
\end{proof}

We are now ready to prove Proposition~\ref{prop-restriction}.
\begin{proof}[Proof of Proposition~\ref{prop-restriction}]
We denote by $\Pb$ the probability measure under which we define $\gamma$ and $K$.  For any $A\in \Qc$, we denote by $\Pb_A$ the image measure of $\Pb$ under the conformal map $f_A$.
Our first goal is to interpret the normalized local martingale $(M_t/ \Eb(M_0))$ as the Radon-Nikodym derivative of $\Pb_A$ with respect to $\Pb$.

Let us fix $A\in\Qc$. Note that $|h_t'(0)| \ge 1$ and  is decreasing in $t$, hence $|h_t'(0)|^{2\mu}$ stays bounded. Also, we have $|h_t'(e^{iW_t})|\le 1$ and $|h_t'(e^{iV_t})| \le 1$. Since $(6-\kappa)/(2\kappa)>0$ and $\nu\ge 0$, the  terms $|h_t'(e^{iW_t})|^{(6-\kappa)/(2\kappa)}$ and $|h_t'(e^{iV_t})|^\nu$ are also bounded.

\begin{itemize}
\item If $b\not=0$, then  $G$ is increasing to $\infty$ as $\theta\to \pi$ by~\eqref{asym-G-pi}. Note that we always have $\nu_t \le \theta_t$ (see Figure~\ref{commutation}). Also note that $G(\theta)>0$ for all $\theta\in(0,\pi)$ by Lemma~\ref{lem:G>0}, $G$ is asymptotic to $\theta^{2d}$ as $\theta\to 0$ due to~\eqref{simplified-G}. 
Altogether, we deduce that the term $G(\nu_t)/ G(\theta_t)$ is bounded.

\item If $b=0$, then we are in the degenerate cases (ia) and (ib). The curve $\gamma$ is a radial SLE$_\kappa(\rho)$. We have $G(\theta) =(\sin\theta)^{2d}$. Since $\nu_t\le \theta_t$, the term $G(\nu_t)/ G(\theta_t)$ can possibly explode only if $\theta_t$ tends to $\pi$.
For $\delta>0$, let 
\begin{align}\label{eq:S_delta}
S_\delta:=\inf\{t>0: \theta_t \ge \pi-\delta\}.
\end{align}
Note that $S_\delta$ tends to $T=\infty$ as $\delta\to 0$.
We remark that if $\kappa=8/3$, then $M_t$ coincides with the local martingale considered in \cite[Lemma 14]{MR3293294}. In \cite[Proof of Proposition 15]{MR3293294}, Wu argued that the local martingale $M_t$ is bounded by $1$, and then used $\Eb[M_T]=M_0$. However, this is in fact not true, and we can construct cases where $M_t>1$. It also does not seem obvious for us to show that it is bounded. We will therefore argue differently. 

\end{itemize}
We still need to consider the exponential term in $M_t$. This term can possibly explode if $\gamma$ hits $A$. However, if we impose that the distance between $\gamma$ and $A$ is at least some $\eps>0$, then this exponential term stays bounded, due to~\eqref{eq:loop_mass_t} and the fact that the mass of loops $m_\Ub(\gamma([0,t]), A)$ is bounded (because the Brownian loop measure has finite mass on loops in $\Ub$ with diameter at least $\eps$ \cite{MR2045953}).

For any $\eps>0$,  let $\tau_\eps$ be the first time that $\gamma$ reaches the $\eps$-neighborhood of $A$. 
If $b\not=0$, then let $S_\delta=T$ for all $\delta>0$. If $b=0$, then let $S_\delta$ be given by~\eqref{eq:S_delta}.
It follows from the arguments above that $(M_t, 0\le t\le \tau_\eps\wedge S_\delta)$ is bounded.
Furthermore, if we view $M$ as a function of the curve $\gamma$, then for any $m>0$, there exists $\eps, \delta>0$ such that for any curve $\gamma$ which is continuous and such that $d(\gamma,A)\ge \eps$, we have $M_{t\wedge S_\delta}(\gamma)\le m$ for all $t\ge 0$. Moreover, we can choose $\eps, \delta$ in a way that $\eps, \delta$ go to $0$ as $m$ goes to $\infty$.

For any $t>0$, let $\Pb_t$ be $\Pb$ restricted to $\Fc_t$. We also let $\Pb_{\tau_\eps \wedge S_\delta}$ be $\Pb$ restricted to $\Fc_{\tau_\eps \wedge S_\delta}$.
Note that $M_{\tau_\eps\wedge S_\delta \wedge t}/\Eb(M_0)$ is a bounded martingale with expectation $1$. By Girsanov's theorem and~\eqref{eq:dM/M}, weighting $\Pb_{\tau_\eps \wedge S_\delta}$ by $M_{\tau_\eps \wedge S_\delta}/\Eb(M_0)$ gives rise to a new probability measure $\wt\Pb_{\tau_\eps \wedge S_\delta}$  on $\Fc_{\tau_\eps \wedge S_\delta}$ such that  for $0\le t\le \tau_\eps \wedge S_\delta$, the driving function $W_t$ of $\gamma$ satisfies
\begin{align}\label{driving_function_girsanov}
d W_t =\sqrt{\kappa} d\wt B_t +\!\left(\frac{6-\kappa}{2}\frac{\phi_t''(W_t)}{\phi_t'(W_t)} +\frac{\kappa}{2} \phi_t'(W_t) \frac{G'(\nu_t)}{G(\nu_t)} \right) dt,
\end{align}
where $\phi_t(z)=-i\ln h_t(e^{iz})$ (see Section~\ref{A:ito}) and $\wt B_t$ is a Brownian motion under $\wt\Pb_{\tau_\eps \wedge S_\delta}$ (recall that under $\Pb_{\tau_\eps \wedge S_\delta}$,  the driving function $W_t$ of $\gamma$ is given by~\eqref{driving-function}). 
As $\eps$ tends to $0$, $\tau_\eps \wedge S_\delta$ increases to $\tau \wedge S_\delta$ and the measures $\wt\Pb_{\tau_\eps \wedge S_\delta}$ are consistent, hence we can obtain a probability measure $\wt \Pb_{\tau \wedge S_\delta}$ on $\Fc_{\tau \wedge S_\delta}$ which coincides with $\wt \Pb_{\tau_\eps \wedge S_\delta}$ on $\Fc_{\tau_\eps \wedge S_\delta}$ for all $\eps>0$. 
Then, letting $\delta$ tend to $0$, we can further obtain a probability measure $\wt \Pb_{\tau \wedge T}$ on $\Fc_{\tau \wedge T}$ which coincides with $\wt \Pb_{\tau_\eps \wedge S_\delta}$ on $\Fc_{\tau_\eps \wedge S_\delta}$ for all $\eps, \delta>0$.
Under $\wt\Pb_{\tau \wedge T}$, $(\gamma_t, 0\le t < \tau\wedge T)$ is an SLE with driving function $ W_t$ given by~\eqref{driving_function_girsanov}.

On the other hand, by Lemma~\ref{lem:change_measure_conformal}, a radial SLE in $\Ub$ driven by~\eqref{driving_function_girsanov} has the same law as the image under $f_A^{-1}$ of an hSLE$_\kappa(\mu,\nu)$ in $\Ub$. In particular, it a.s.\ does not intersect $A$.
Therefore,  $\wt\Pb_{\tau\wedge T}$ is in fact equal to $\Pb_{A}$. Moreover, under $\wt\Pb_{\tau \wedge T}$, we have $\tau=\infty$ a.s.
Hence we have proved
\begin{align}\label{eq:RN}
d\Pb_A(\gamma)\mathbf{1}_{\gamma \cap A=\emptyset}=\frac{M_T}{\Eb(M_0)}d\Pb(\gamma)\mathbf{1}_{\gamma\cap A=\emptyset}.
\end{align}
\medbreak

If we are in situation (ia), (ii) or (iiia), then Lemma~\ref{lem:limitMT}, equation \eqref{eq:RN} and the fact that 
\begin{align}\label{eq:M0}
\Eb(M_0)=|f_A'(0)|^\alpha f_A'(1)^\beta
\end{align}
imply that $\gamma$ (hence also $K$) satisfies the formula~\eqref{eq:gen} and we have completed the proof. 

If we are in situation (ib) or (iiib), then again by Lemma~\ref{lem:limitMT} and \eqref{eq:RN}, \eqref{eq:M0}, we have
\begin{align}\label{eq:gamma'}
|f_A'(0)|^\alpha f_A'(1)^\beta d\Pb_A(\gamma)\mathbf{1}_{\gamma \cap A=\emptyset}= \exp\!\left(-c(\kappa) m_{\Ub}(\gamma,A) \right)  d\Pb(\gamma)\mathbf{1}_{\gamma\cap A=\emptyset} C^\nu.
\end{align}
Multiplying~\eqref{eq:gamma'} by $\mathbf{1}_{\gamma'\cap A=\emptyset}$ and applying~\eqref{eq:C}, we get
\begin{align*}
&|f_A'(0)|^\alpha f_A'(1)^\beta d\Pb_A(\gamma) d\Pb_A^\gamma(\gamma') \mathbf{1}_{(\gamma\cup\gamma') \cap A=\emptyset}\\
=&\exp\!\left(- c(\kappa) \!\left(m_{\Ub}(\gamma,A) + m_{D_\gamma}(\gamma', A) \right) \right)  d\Pb(\gamma) d\Pb^\gamma(\gamma')  \mathbf{1}_{(\gamma\cup\gamma') \cap A=\emptyset}.
\end{align*}
The equation above coincides with~\eqref{eq:gen}, because the event $(\gamma\cup\gamma') \cap A=\emptyset$ is the same as $K \cap A=\emptyset$, and  that
\[ d\Pb_A(\gamma) d\Pb_A^\gamma(\gamma') = d\Pb_A(K), \quad  d\Pb(\gamma) d\Pb^\gamma(\gamma') = d\Pb(K).
\]
Also note that $m_{\Ub}(\gamma,A)+ m_{D_\gamma}(\gamma', A) =m_\Ub(K, A)$. This completes the proof.
\end{proof}

\section{Proof of Theorem~\ref{main-theorem}}\label{sec-disconnection}
In the present section, we aim to prove Theorem~\ref{main-theorem}. We will first recall in Section~\ref{S:E} some results on eigenvalue expansions for diffusion hitting times. Then in Section~\ref{S:P}, we will prove Theorem~\ref{main-theorem} by analysing the diffusion process $\theta_t$ related to the hSLE defined in Section~\ref{ShSLE}.

\subsection{Eigenvalue expansions for diffusion hitting times}\label{S:E}
In this section, we will recall some results on eigenvalue expansions for diffusion hitting times, based on classical diffusion theory and a result of Kent~\cite{MR576891}. 

Consider a diffusion process $X$ defined on an interval $[r_0, r_1]$ where $-\infty\le r_0 < r_1 \le \infty$. Suppose that $X$ is associated with an infinitesimal generator $\Lc$ which is a second order linear operator. By classical diffusion theory (see, e.g., \cite{MR0345224,MR0247667}), the diffusion process $X$ can be characterised  by the \emph{speed measure} $m(x)$, \emph{natural scale} $s(x)$ and \emph{killing measure} $k(x)$.
We use Mandl's terminology for the boundaries: \emph{regular, entrance, exit} and \emph{natural} (see \cite{MR0247667}).  For the purpose of the present paper, we restrict ourselves to the case where the boundary condition at $r_0$ is \emph{entrance}.
For a continuous function $u$ defined on $(r_0, r_1)$, we  define its right-hand derivative with respect to the natural scale to be
\begin{align*}
u^+(x)=\lim_{\eps\to 0} (u(x+\eps) -u(x)) / (s(x+\eps) -s(x)),
\end{align*} 
when the limit above does exist.

For $r_0\le a < b \le r_1$, let $\tau_{a,b}$ be the first time that $X$ hits $b$, starting at $a$. Let $\phi_{a,b}$  be the moment generating function of $\tau_{a,b}$, given by $\phi_{a,b}(\lambda) = \Eb (\exp(\lambda \tau_{a,b}))$.
For fixed $\lambda$, let $v_\lambda$ be a solution to
\begin{align*}
\Lc v+ \lambda v=0  \quad \text{with initial condition}\quad v^+(r_0)=0.
\end{align*}
The solution above is unique up to multiplicative constant. We can fix it by imposing $v_\lambda(r_0)=1$. Then, by  \cite{MR0345224}, we have
\begin{align*}
\phi_{a,b}(\lambda)= v_\lambda(a)/v_\lambda(b).
\end{align*}
Based on the above observations, Kent further deduced the following theorem on the eigenvalue expansions for $\tau_{a,b}$, that we state in a form which suits our purpose.

\begin{theorem}[\cite{MR576891}]\label{thm:kent}
For $r_0\le a < b \le r_1$, the zeros of $v_\lambda(b)$, viewed as a function of $\lambda$, are simple and positive and form a sequence
$0<\lambda_1<\lambda_2<\cdots.$

For all $t>0$, we have
\vspace{-2mm}
\begin{align*}
\Pb(\tau_{a,b}>t) = \sum_{n=1}^\infty a_n \exp(-\lambda_{n} t),
\end{align*}
\vspace{-2mm}
where
\vspace{-2mm}
\begin{align*}
a_n=\prod_{k=0,k\not=n}^\infty (1-\lambda_n/\lambda_k)^{-1}.
\end{align*}
\end{theorem}

\subsection{Proof of Theorem~\ref{main-theorem}}\label{S:P}
We are now ready to prove Theorem~\ref{main-theorem}.

For $\kappa\in(0,4]$ and $\alpha,\beta$ in the range (\ref{range1}), let $K$ be a radial $\kappa$-restriction measure with parameters $(\alpha,\beta)$. Let $K_0$ be the connected component containing the origin of the interior of $K$.
Let $L$ be the conformal radius of $K_0$ seen from the origin.  
In Section~\ref{SS restriction}, we have constructed $K$ using an hSLE curve $\gamma$, which is in turn parametrized in terms of the conformal radius of its complement. Therefore, we in fact have $L=e^{-T}$ where $T$ is the stopping time defined in~\eqref{stoping_time} ($T$ is also the first time that $\gamma$ disconnects $0$ from $\infty$).

Therefore, proving Theorem~\ref{main-theorem} boils down to estimating the tail probability of $T$. Note that $T$ is the first hitting time at $\pi$ by the diffusion process $\theta_t$ started at $0$ and governed by the following equation
\begin{align}\label{eq-theta}
 d \theta_t = \frac12 \sqrt{\kappa} dB_t +\frac{\kappa}{4 }\frac{G'(\theta_t)}{G(\theta_t)} dt +\frac12 \cot(\theta_t) dt.
\end{align}
We have already argued in Section \ref{S:geom} that $\theta_t$ a.s.\ never hits $0$ (except at $t=0$) and will a.s.\ hit $\pi$. 
We will apply Theroem~\ref{thm:kent} and explicitly compute the relevant eigenvalues.
More precisely, we will prove the following lemma, which then implies Theorem~\ref{main-theorem}.
\begin{lemma}
For any $t>0$, we have
\begin{align}\label{T>t}
\Pb(T>t)=\sum_{n=0}^\infty a_n \exp(-\lambda_n t).
\end{align}
where $(\lambda_n)_{n\in\Nb}$ is a positive increasing sequence given by
\begin{align}\label{eq:lambda_n}
\lambda_n=\left(n^2+n-\frac12 \right)\frac{\kappa}{8}-\frac{n-1}{2}-\frac{1}{\kappa}+\frac{\beta}{2}+\left(\frac18\left(n+\frac12 \right)-\frac{1}{4\kappa} \right)\sqrt{16\kappa\beta+(4-\kappa)^2}-\alpha,
\end{align}
and
\vspace{-4mm}
\begin{align*}
a_n=\prod_{k=0,k\not=n}^\infty (1-\lambda_n/\lambda_k)^{-1}.
\end{align*}
\end{lemma}

\begin{proof}
The infinitesimal generator $\Lc$ of the process $\theta_t$ is given by
\begin{align*}
\Lc f(\theta)=\frac{\kappa}{8}f''(\theta)+\frac{\kappa}{4}\frac{G'(\theta)}{G(\theta)} f'(\theta) +\frac12 \cot(\theta)f'(\theta).
\end{align*}
The natural scale $s(\theta)$ associated to $\theta$ satisfies
\begin{align*}
\Lc s=0.
\end{align*}
Equivalently, the function $h:=s'$ satisfies the following first order linear differential equation
\begin{align*}
h'(\theta)/h(\theta)=-2\frac{G'(\theta)}{G(\theta)}-\frac{4}{\kappa} \cot(\theta).
\end{align*}
Integrating both sides, we deduce that for some $c>0$, we have
\begin{align*}
s'(\theta)=h(\theta)= c  \sin(\theta)^{-4/\kappa} G(\theta)^{-2}.
\end{align*}
Since $G(\theta) \sim \theta^{2d}$ as $\theta\to 0$, we deduce that
\begin{align*}
s(\theta)-s(0)= c\, \theta^{1-4d-4/\kappa} (1+o(1)) \text{ as } \theta\to 0.
\end{align*}
For $\lambda>0$, let $f_\lambda$ be the unique solution to the following differential equation
\begin{align}\label{h_lambda}
\Lc f +\lambda f=0
\end{align}
with initial conditions
\begin{align}\label{eq:initial1}
&\lim_{\theta\to 0} (f(\theta)-f(0))/(s(\theta)-s(0))=\lim_{\theta\to 0} \theta^{-1+4d+4/\kappa}(f(\theta)-f(0))=0\\
&f_\lambda(0)=1. \label{eq:initial2}
\end{align}
It then remains to solve (\ref{h_lambda}) and to identify the sequence of successive zeros (in increasing order) of $f_\lambda(\pi)$ viewed as a function of $\lambda$.

To solve (\ref{h_lambda}), we try to find a solution in the form 
\begin{align}\label{fgl}
f_\lambda(\theta)=g_\lambda(\theta)/G(\theta),
\end{align}
where $G$ is given by Definition~\ref{def:G}. This implies that $g_\lambda$ should be a solution of
\begin{align}\label{diff-G-lambda}
e+\lambda-\frac{\nu}{2\sin(\theta)^2} +\frac{g'(\theta)}{g(\theta)} \frac{\cot(\theta)}{2}+\frac{\kappa}{8} \frac{g''(\theta)}{g(\theta)} =0.
\end{align}
Note that (\ref{diff-G-lambda}) has the same form as (\ref{diff-G}), except that we replace $e$ by $\wt e=e+\lambda$. 
This equation can be transformed into a hypergeometric differential equation after a change of variable (as in the proof of Lemma~\ref{lem3.2}) and has two linearly independent solutions. The initial condition~\eqref{eq:initial2} implies that $g_\lambda$ should have the same form as (\ref{def1}).
More concretely, $g_\lambda$ is given by~\eqref{def1} where we replace the parameters $a,b$ by the following $a_\lambda,b_\lambda$ 
\begin{align*}
 a_\lambda=\frac14+q_1(\kappa, \mu+\lambda/2)+q_2(\kappa,\nu),\quad
 b_\lambda=\frac14-q_1(\kappa, \mu+\lambda/2)+q_2(\kappa,\nu)
\end{align*}
and leave  $c,d$ invariant. 
More precisely, we have
\begin{align}\label{f-lambda1}
f_\lambda(\theta)= {_2}F_1 (a_\lambda, b_\lambda; c; \sin(\theta)^2)/ {_2}F_1 (a,b;c; \sin(\theta)^2), \quad \text{for } \theta\in(0,\pi/2),
\end{align}
and that $f_\lambda$ is the analytical continuation of~\eqref{f-lambda1} for $\theta\in [\pi/2,\pi)$.
We can then check that $|f_\lambda'(0)|<\infty$, hence the condition~\eqref{eq:initial1} is also satisfied.
Therefore, $f_\lambda$ given by \eqref{f-lambda1} is indeed the unique solution to~\eqref{h_lambda} which satisfies the initial conditions~(\ref{eq:initial1}, \ref{eq:initial2}).

In order to apply Theorem~\ref{thm:kent}, we need to compute the value of $f_\lambda(\theta)$ as $\theta\to\pi$.
Recall that by~\eqref{asym-G-pi}, we have that as $\theta\to \pi$,
\begin{align*}
G(\theta) \sim -2 C_2(a,b,c) (\pi-\theta)^{2d+2-2c},
\end{align*}
where $C_2$ is given by~\eqref{c2}, i.e.,
\begin{align*}
C_2(a,b,c)=\frac{\Gamma(c)\Gamma(-1/2)}{\Gamma(a)\Gamma(b)}\cdot\frac{\Gamma(3/2)\Gamma(c-1)}{\Gamma(c-a)\Gamma(1/2+a)}.
\end{align*}
For the same reasons, we have that  as $\theta\to \pi$,
\begin{align*}
g_\lambda(\theta) \sim -2 C_2(a_\lambda,b_\lambda,c) (\pi-\theta)^{2d+2-2c}.
\end{align*}
Therefore
\begin{align*}
f_\lambda(\pi)=C_2(a_\lambda, b_\lambda, c)/C_2(a,b,c),
\end{align*}
Note that for $\alpha,\beta$ in the range~\eqref{range1}, since $b\not=0$, we have that $C_2(a,b,c)\not=0$.
The function $C_2(a_\lambda, b_\lambda, c)$ equals zero if and only if its denominator
\begin{align*}
\Gamma(a_\lambda) \Gamma(b_\lambda) \Gamma(c-a_\lambda) \Gamma(a_\lambda+1/2)
\end{align*}
is $\infty$. The Gamma function equals $\infty$ if and only if its argument is in $\Zb_-$. For $\mu,\nu$ in the range~\eqref{eq:range} (but we rule out $b=0$) and $\lambda\ge 0$, only $b_\lambda$ and $c- a_\lambda$ can possibly belong to $\Zb_-$. The sequence of $(\lambda_n)_{n\in\Nb}$ that makes $b_\lambda \in \Zb_-$ or $c- a_\lambda\in \Zb_-$ is given by
\begin{align*}
q_1(\kappa, \mu+\lambda_n/2)=\frac{2n+1}{4} +q_2(\kappa,\nu).
\end{align*}
Solving this equation yields to
\begin{align*}
\lambda_n=&\left(n+\frac12 \right)^2\frac{\kappa}{8}+\frac{\nu}{2}-\frac{3}{32\kappa}(4-\kappa)^2-2\mu
+\left(n+\frac12 \right)\frac{1}{8}\sqrt{16\kappa\nu+(4-\kappa)^2},
\end{align*}
which is equal to~\eqref{eq:lambda_n}, by~\eqref{alpha-beta}. This completes the proof.
\end{proof}

\section{Some further remarks}\label{sec:further}
Up until now, we have completed the proofs of all the main results of this article. In this section, we will make some additional remarks.

\subsection{Remarks on generalized intersection exponents}\label{intro:intersection}
Recall that the half-plane and whole-plane Brownian intersection exponents were established in \cite{MR1879850, MR1879851, MR1899232}, and are closely related to the Brownian disconnection exponent. 
We believe that one can also generalize the intersection exponents to all $\kappa\in(0,4]$, and relate these exponents to (other forms of) general restriction measuers.

In Section~\ref{sec:further_restriction}, we will define and study the chordal and trichordal general $\kappa$-restriction measures for $\kappa\in(0,4]$. In the study of general trichordal restriction measures, the following exponents will show up. (More precisely, these exponents come from the equations of the chordal hypergeometric SLEs used to construct the trichordal $\kappa$-restriction measures.)
For $\kappa\in(0,4]$ and $x>0$, let
\begin{align}\label{V-kappa}
V_{\kappa}(x)=\sqrt{16\kappa x+(4-\kappa)^2}-(4-\kappa), \quad V_\kappa^{-1}(x)=\frac{x^2+2(4-\kappa)x}{16\kappa}.
\end{align}
For $n\in\Nb$ and $\alpha_1, \cdots, \alpha_n >0$, let
\begin{align}\label{tilde-xi-c}
&\wt\xi_{\kappa}(\alpha_1,\cdots, \alpha_n)=V_\kappa^{-1}\left(V_\kappa(\alpha_1)+\cdots +V_\kappa(\alpha_n) \right)\\
\label{xi-c}
&\xi_{\kappa}(\alpha_1,\cdots,\alpha_n)=\frac{1}{32\kappa}\left(V_\kappa(\alpha_1)+\cdots +V_\kappa(\alpha_n) \right)^2
-\frac{1}{8\kappa}(4-\kappa)^2.
\end{align}
One can see that $\wt\xi_\kappa$ and $\xi_\kappa$ coincide with the Brownian intersection exponents when $\kappa=8/3$, and they satisfy the same cascade relations as the ones satisfied by the Brownian intersection exponents (see e.g.\ \cite{MR1899232}):
\begin{align*}
\wt\xi_{\kappa}(\alpha_1,\cdots, \alpha_n, \wt\xi_{\kappa}(\beta_1,\cdots, \beta_k))=\wt\xi_{\kappa}(\alpha_1,\cdots, \alpha_n, \beta_1,\cdots, \beta_k),\\
\xi_{\kappa}(\alpha_1,\cdots, \alpha_n, \wt\xi_{\kappa}(\beta_1,\cdots, \beta_k))=\xi_{\kappa}(\alpha_1,\cdots, \alpha_n, \beta_1,\cdots, \beta_k).
\end{align*}
These evidences suggest, at a heuristic level, that the exponents~\eqref{tilde-xi-c} and~\eqref{xi-c} should respectively be the half-plane and whole-plane generalized intersection exponents. We mention that the exponents~\eqref{tilde-xi-c} and~\eqref{xi-c} have also appeared in \cite{MR2112128, MR2581884} in the context of quantum gravity and the KPZ relation.
However, it remains an interesting question to establish the physical meaning of these exponents in terms of non-intersection probability between random sets, so as to show that they are indeed the ``correct'' generalization of Brownian intersection exponents.

\subsection{General chordal and trichordal restriction measures}\label{sec:further_restriction}
The goal of this section is to write down a more complete theory on general restriction measures. More concretely, we will describe the chordal and trichordal counterparts of the general radial restriction measures that we considered in this work.  As we mentioned earlier, the exponents~\eqref{tilde-xi-c} and~\eqref{xi-c} will show up in the study of the general trichordal restriction measures.

The construction of  the general chordal and trichordal restriction measures (for $\kappa\in(0,4]$) is in fact much more straightforward than in the radial case. It is known that the chordal restriction measures can be constructed by chordal SLE$_{8/3}(\rho)$ processes (see \cite{MR1992830, MR2060031}) and the trichordal restriction measures can be constructed by chordal hSLE$_{8/3}(\mu, \nu, \lambda)$ processes (see \cite{MR3827221}). Therefore, we can simply change the value of $\kappa$ in the SLEs to construct the corresponding general restriction measures.
Recall that radial restriction measures were constructed in \cite{MR3293294}  with a method (see discussion below Theorem~\ref{intro-thm:restriction}) which cannot be generalized to other $\kappa$.

\subsubsection{General chordal restriction measures}
Let us first give the definition of general chordal restriction measures. 
Let $\Omega_1$ be the collection of all  simply connected compact sets $K\subset \overline\Ub$ such that $K\cap \partial \Ub=\{-1, 1\}$.
Let $\Qc_1$ be the collection of all compact sets $A\subset\overline\Ub$ such that $\Ub\setminus A$ is simply connected and $-1,1\not\in A$. 
For all $A\in\Qc_1$, let $\wt f_A$ be a conformal map from $\Ub\setminus A$ onto $\Ub$ that leaves $-1,1$ fixed.

\begin{definition}[General chordal restriction measure]\label{def:gen_chordal}
For $\kappa\in(0,4]$ related to $c$ by~\eqref{kappa-c}, a probability measure $\Pb$ on $\Omega_1$ (or a random set $K$ of law $\Pb$) is said to satisfy chordal $\kappa$-restriction with exponents $\alpha \in\Rb$, if for all $A\in\Qc_1$, we have
\begin{align}\label{eq:gen_chordal}
\frac{d\Pb(K)}{d\Pb_A(K)}\mathbf{1}_{K\cap A=\emptyset}=\mathbf{1}_{K\cap A=\emptyset} \left(\wt f_A'(1) \wt f_A(-1)\right)^\alpha \exp\left(c(\kappa) \, m_{\Ub}(K,A)\right),
\end{align}
where $\Pb_A=\Pb\circ \wt f_A$ and $m_{\Ub}(K,A)$ is the mass under the Brownian loop measure of all loops in $\Ub$ that intersect both $K$ and $A$.
\end{definition}
For each $\alpha\in\Rb$, the uniqueness of measures that satisfy Definition~\ref{def:gen_chordal} remains to be proved (except for $\kappa=8/3$). In the following Proposition~\ref{intro-thm:restriction_chordal}, we will prove the existence of such measures for a certain range of $\alpha$.

We remark that there is also a one-sided version of chordal $\kappa$-restriction measures if we restrict to $A\in\Qc_1$ such that $A\cap \partial \Ub$ is a subset of the lower half circle. It is known that chordal SLE$_\kappa(\rho)$ curves satisfy one-sided chordal $\kappa$-restriction with parameter (see  \cite{MR1992830})
\begin{align}\label{eq:alpha-chordal}
\alpha(\kappa, \rho)=(\rho+2)(\rho+6-\kappa)/(4\kappa).
\end{align}
The range $\rho\in(-2,\infty)$ gives rise to the range $\alpha\in(0,\infty)$.
Let us now record a result for the (two-sided) chordal $\kappa$-restriction measures.

\begin{figure}[h!]
\centering
\includegraphics[width=.73\textwidth]{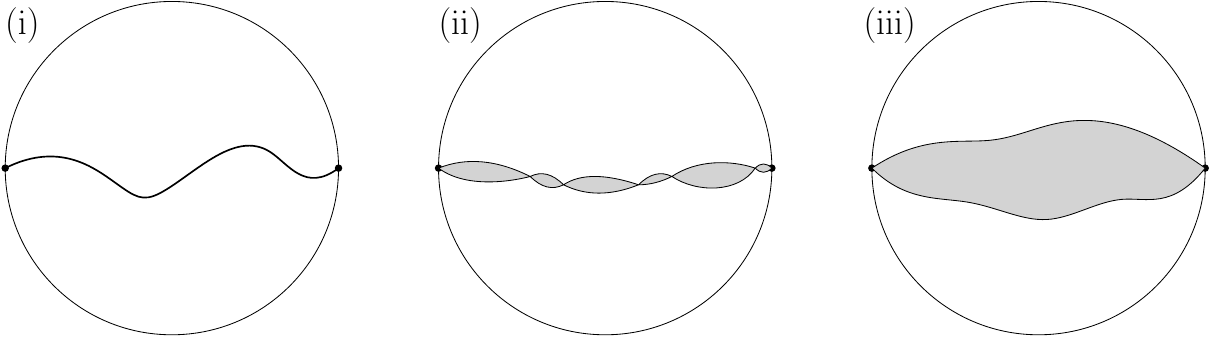}
\caption{Geometric properties of general chordal restriction measures}
\label{fig:chordal}
\end{figure}

\begin{proposition}\label{intro-thm:restriction_chordal}
For any $\kappa\in(0,4]$, and  $\alpha$ in the following range
\begin{align}\label{eq:thm_range_chordal}
\alpha \ge (6-\kappa)/(2\kappa),
\end{align}
there exists a measure $\Pf_\kappa^{\alpha}$ which satisfies chordal $\kappa$-restriction with parameter $\alpha$.
Moreover, if $K$ is a sample with law $\Pf_\kappa^{\alpha}$, then it satisfies the following geometric properties:
\begin{enumerate}[(i)]
\item If $\alpha=(6-\kappa)/(2\kappa)$, then $K$ is given by a chordal SLE$_\kappa$ curve, hence is a.s.\ a simple curve, see Figure~\ref{fig:chordal}(i). 

\item If  $\alpha\in ((6-\kappa)/(2\kappa),  (12-\kappa)(\kappa+4)/(16\kappa))$, then $K$ is a.s.\ not a simple curve, but has cut-points, see Figure~\ref{fig:chordal}(ii).

\item If $\alpha\ge  (12-\kappa)(\kappa+4)/(16\kappa)$, then $K$ a.s.\ does not have cut-points, see Figure~\ref{fig:chordal}(iii).
\end{enumerate}
\end{proposition}

\begin{proof}
The case (i) is a known property of chordal SLE$_\kappa$ curves  \cite{MR1992830}. For $\alpha > (6-\kappa)/(2\kappa)$, we will use the same method of construction as in \cite[Section 5.2]{MR2060031} for the $\kappa=8/3$ case.
We first construct the lower boundary of $K$ using a chordal SLE$_\kappa(\rho)$ curve $\eta_1$ in $\Ub$ from $1$ to $-1$ with marked point immediately to the right of $1$, where $\rho$ is related to $\alpha$ by~\eqref{eq:alpha-chordal}.
Note that $\alpha > (6-\kappa)/(2\kappa)$ corresponds to $\rho>0$.
Then, given $\eta_1$, we can construct the upper boundary of $K$ as a SLE$_\kappa(\rho-2)$ curve $\eta_2$ in the connected component $\Ub\setminus \eta_1$ which is above $\eta_1$, from $1$ to $-1$ with marked point immediately to the left of its starting point.
The compact set $K$ bounded by the two curves then satisfies chordal $\kappa$-restriction with parameter $\alpha$. The proof of the last fact is completely analogous to \cite[Section 5.2]{MR2060031}, namely we use the same restriction martingale as there, by changing $\kappa=8/3$ to general $\kappa$. For the sake of brevity, we leave the details to the interested reader.

The geometric property of $K$ is determined by whether $\eta_2$ hits $\eta_1$. Since $\eta_2$ is a SLE$_\kappa(\rho-2)$ curve, it a.s.\ hits the boundary (hence $\eta_1)$ if and only if $\rho<\kappa/2$, by \cite[Lemma 8.3]{MR1992830}. Putting this back into~\eqref{eq:alpha-chordal} completes the proof.
\end{proof}
We believe that~\eqref{eq:thm_range_chordal} is the maximal range for which chordal $\kappa$-restriction measures with parameter $\alpha$ exist. It was proved to be the case for $\kappa=8/3$ \cite{MR1992830}, but it remains to be proved for other $\kappa\in(0,4]$.

\subsubsection{General trichordal restriction measures}
Let us now define the general trichordal restriction measures. Fix $x,y,z \in \partial \Ub$. 
Let $\Omega_2$ be the collection of all  simply connected compact sets $K\subset \overline\Ub$ such that $K\cap \partial \Ub=\{x, y, z\}$.
Let $\Qc_2$ be the collection of all compact sets $A\subset\overline\Ub$ such that $\Ub\setminus A$ is simply connected and $x, y, z \not\in A$. 
For all $A\in\Qc_2$, let $\wh f_A$ be a conformal map from $\Ub\setminus A$ onto $\Ub$ that leaves $x, y, z$ fixed.

\begin{definition}[General trichordal restriction measure]\label{def:gen_trichordal}
For $\kappa\in(0,4]$ related to $c$ by~\eqref{kappa-c}, a probability measure $\Pb$ on $\Omega_2$ (or a random set $K$ of law $\Pb$) is said to satisfy trichordal $\kappa$-restriction with exponents $(\alpha, \beta,\gamma) \in\Rb^3$, if for all $A\in\Qc_2$, we have
\begin{align}\label{eq:gen_trichordal}
\frac{d\Pb(K)}{d\Pb_A(K)}\mathbf{1}_{K\cap A=\emptyset}=\mathbf{1}_{K\cap A=\emptyset} \wh f_A'(x)^\alpha \wh f_A(y)^\beta \wh f_A(z)^\gamma \exp\left(c(\kappa) \, m_{\Ub}(K,A)\right),
\end{align}
where $\Pb_A=\Pb\circ \wh f_A$ and $m_{\Ub}(K,A)$ is the mass under the Brownian loop measure of all loops in $\Ub$ that intersect both $K$ and $A$.
\end{definition}
For each $(\alpha, \beta,\gamma) \in\Rb^3$, the uniqueness of measures that satisfy Definition~\ref{def:gen_trichordal} remains to be proved (except for $\kappa=8/3$). In the following Proposition~\ref{prop:restriction_trichordal}, we will prove the existence of such measures for a certain range of $\alpha, \beta, \gamma$.

\begin{figure}[h!]
\centering
\includegraphics[width=\textwidth]{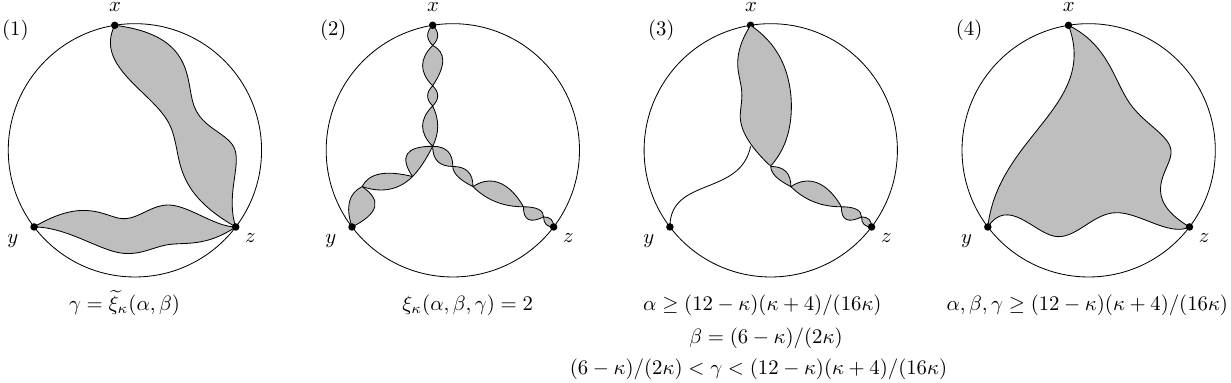}
\caption{Geometric properties of general trichordal restriction measures}
\label{fig:trichordal}
\end{figure}

\begin{proposition}\label{prop:restriction_trichordal}
For any $\kappa\in(0,4]$, and $\alpha, \beta, \gamma$ in the following range
\begin{enumerate}[(a)]
\item $\alpha, \beta,\gamma \ge (6-\kappa)/(2\kappa)$,
\item $\alpha \le \wt \xi_\kappa(\beta, \gamma), \,\, \beta \le \wt \xi_\kappa(\alpha, \gamma), \,\, \gamma \le  \wt \xi_\kappa(\alpha, \beta)$,
\item $\xi_\kappa(\alpha, \beta, \gamma) \le 2$
\end{enumerate}
there exists a measure $\Pf_\kappa^{\alpha,\beta,\gamma}$ which satisfies trichordal $\kappa$-restriction with parameters $\alpha, \beta, \gamma$.
Moreover, if $K$ is a sample with law $\Pf_\kappa^{\alpha, \beta, \gamma}$, then it satisfies the following geometric properties:
\begin{enumerate}[(i)]
\item If $\alpha = \wt \xi_\kappa(\beta, \gamma)$ (resp.\ $\beta = \wt \xi_\kappa(\alpha, \gamma)$, $\gamma =  \wt \xi_\kappa(\alpha, \beta)$), then the boundary point $x$ (resp.\ $y,z$) is a cut-point of $K$. See Figure~\ref{fig:trichordal} (1).

\item If $\xi_\kappa(\alpha, \beta, \gamma) = 2$, then $K$ has a triple disconnecting point in the middle. See Figure~\ref{fig:trichordal} (2).
\item For $\alpha < \wt \xi_\kappa(\beta, \gamma)$ (resp.\ $\beta < \wt \xi_\kappa(\alpha, \gamma)$, $\gamma <  \wt \xi_\kappa(\alpha, \beta)$), the local geometric property of $K$ in a small neighborhood of $x$ (resp.\ $y,z$) is the same as that of a chordal $\kappa$-restriction measure with parameter $\alpha$ (resp.\ $\beta, \gamma$). See Figure~\ref{fig:trichordal}.

%
%

\end{enumerate}
\end{proposition}
\begin{proof}
We will use the same method of construction as in \cite{MR3827221}. The case (i) is a limiting special case, for which the boundaries of $K$ can be constructed using chordal SLE$_\kappa(\rho_1, \rho_2)$ curves.
For all the other cases, we can first construct the part of boundary between $x$ and $y$, using a chordal hypergeometric SLE curve $\eta_1$. Then given $\eta_1$, in the connected component of $\Ub\setminus\eta_1$ which contains $z$ on the boundary, we construct the part of boundary from $y$ to $z$, using another  chordal hypergeometric SLE curve $\eta_2$. Finally, given $\eta_1, \eta_2$, we construct the boundary between $x$ and $z$ using a third  chordal hypergeometric SLE curve in the appropriate connected component of $\Ub\setminus(\eta_1\cup\eta_2)$. 

The proofs of the above facts are completely analogous to \cite{MR3827221}, namely we only need to change the value of $\kappa$. The exponents $\wt\xi_\kappa$ and $\xi_\kappa$ in the range emerge from the equations of the chordal hSLE and the associated restriction martingale. For the sake of brevity, we leave the details to the interested reader.

The geometric properties of $K$ can be deduced from the boundary-hitting behavior of the chordal hypergeometric SLEs. Such properties were written down in  \cite[Section 4.3]{MR3827221} for $\kappa=8/3$, but can easily be generalized to other $\kappa\in(0,4]$. 
\end{proof}
We believe that the range in Proposition~\ref{prop:restriction_trichordal} is the maximal range for which trichordal $\kappa$-restriction measures with parameters $(\alpha,\beta,\gamma)$ exist. It was proved to be the case for $\kappa=8/3$ \cite{MR3827221}, but it remains to be proved for other $\kappa\in(0,4]$.

\begin{remark}

Let $F$ denote the hypergeometric function ${_2F_1}(a, b; c; \cdot)$ where
\begin{align*}
a= \lambda+\mu+\nu+8/\kappa-1, \quad b= \mu+\nu-\lambda, \quad c= 2\nu+4/\kappa.
\end{align*}
The hypergeometric hSLE$_\kappa(\nu, \mu,\lambda)$ in the upper half-plane from $0$ to $\infty$ with force points $x_2<x_1 <0$ is defined in  \cite{MR3827221} to be a curve $(\gamma_t, t\ge 0)$ generated by the Loewner equation
\begin{align*}
d g_t(z) =\frac{2dt}{g_t(z) - W_t},
\end{align*}
where $g_t$ is the conformal map from the infinite connected component of $\Hb\setminus \gamma([0,t])$ onto $\Hb$ normalized by $g_t(z)=z+2t/z+o(1/z)$. The driving function $W_t$ satisfies
\begin{align}\label{eq:hsle1}
d W_t = \sqrt{\kappa} d B_t + \left( \frac{\kappa \nu}{W_t -O_t} + \frac{\kappa\mu}{W_t -V_t} -\frac{\kappa}{O_t- V_t} \frac{F'(\frac{O_t-W_t}{O_t- V_t})}{F (\frac{O_t-W_t}{O_t- V_t} )} \right) dt,
\end{align}
where $O_t=g_t(x_1)$ and $V_t=g_t(x_2)$.\footnote{The equation~\eqref{eq:hsle1} was written for general $\kappa$ in \cite[(4.21)]{MR3827221}. However, when defining the parameters $a,b, c$ in \cite[(4.2)]{MR3827221}, we directly substituted $\kappa$ with $8/3$, because we were only interested in the $\kappa=8/3$ case there.} 
The intermediate iSLE$(\kappa; \rho)$ defined by Zhan in \cite{MR2646499} is equal to hSLE$_\kappa(\nu, \mu, \lambda)$ (see e.g.\ \cite[Remark 2.17]{MR4235483}) for
$$\mu=\nu=(\rho+2)/\kappa, \quad  \lambda=1-8/\kappa.$$
The fact that the iSLE is a subfamily of hSLE can be seen by comparing the driving functions, but can be more directly and intuitively understood as a consequence of the trichordal restriction property which ``defines'' the hypergeometric SLE: We believe that the family  hSLE$_{\kappa}(\nu, \mu,\lambda)$ is the maximal 
 family of curves that satisfy the trichordal $\kappa$-restriction (this is proved in \cite{MR3827221} for $\kappa=8/3$).  
The iSLE$(\kappa;\rho)$ is the time-reversal of the SLE$_\kappa(\rho)$, hence satisfies trichordal $\kappa$-restriction, so it should belong to the family of hSLE.
Similarly, one can also see that the family of hSLE contains the family of SLE$_\kappa(\rho_1, \rho_2)$ processes. It is easy to work out the correspondence of the parameters in this case:
 \begin{align*}
\nu=\rho_1/\kappa, \quad \mu=-(\rho_2+4)/\kappa+1, \quad \lambda=-(\rho_1+\rho_2+8)/\kappa+1.
\end{align*}
For the same reason, the time-reversal of an SLE$_\kappa(\rho_1, \rho_2)$ with force points $x_
2<x_1=0^-$ should also belong to the hSLE$_{\kappa}(\nu, \mu,\lambda)$ family. For this case, the correspondence of the parameters is computed in \cite[Remark 5.2]{MR4374683}:
\begin{align*}
\nu=(\rho_1+\rho_2)/\kappa, \quad \mu=(\kappa-4-\rho_2)/\kappa, \quad \lambda=(\kappa-8-\rho_1)/\kappa.
\end{align*}
\end{remark}

%
%
%
%

\appendix
\section{More computations about SLE}\label{A1}
In this appendix, we are going to prove some results used in Section~\ref{SS martingale}. We are mostly going to carry out It\^o calculus, taking into account of the coordinate changes under various conformal maps.
The computations are lengthy, but they seem to be unavoidable.  Computations of this type have occurred before in other SLE works (see, e.g., \cite{MR1992830}), but in simpler settings. 

\subsection{It\^o computation for radial SLE}\label{A:ito}
In this section, we are going to carry out It\^o computation for a radial SLE given by~\eqref{radial eq} with driving function $(W_t)$. Our only assumption on $W_t$ is that it is a semi-martingale with martingale part $\sqrt{\kappa} dB_t$. Otherwise, the computation here does not depend on the exact form of $(W_t)$, hence is valid for general forms of radial SLE$_\kappa$ with any $\kappa>0$.

We will mainly look at how the curve transforms under the conformal map $f_A$ for $A\in\Qc$.
Recall that we have defined $A_t=g_t(A)$, $h_t=f_{A_t}$ and $\wt g_t=h_t\circ g_t \circ f_A$ (see Figure~\ref{commutation}).
For all $z\in \Rb$, define $\phi_t(z)=-i\ln h_t(e^{iz})$ where the branch of the logarithm is chosen in a way that $\phi_0(z)=z$ and $\phi_t(z)$ is continuous in $t$. Let $\wt W_t=\phi_t(W_t)$. 
Then we have
\begin{align*}
|h_t' (e^{iW_t})| =\phi_t' (W_t), \quad |h_t'(e^{iV_t})|=\phi_t'(V_t), \quad \nu_t =(\phi_t(W_t)-\phi_t(V_t))/2.
\end{align*}
We aim to derive the It\^o differential of the terms $ h_t(e^{iW_t})$, $ h_t'(e^{iW_t})$ and $\phi_t'(W_t)$.

Since $h_t=\wt g_t\circ f_A \circ g_t^{-1}$, using the chain rule, we get
\begin{align*}
d \wt g_t(z)=\wt g_t(z) \frac{e^{i\wt W_t}+\wt g_t(z)}{e^{i\wt W_t}-\wt g_t(z)} |h_t'(e^{iW_t})|^2 dt.
\end{align*}
Knowing that $h_t'(e^{iW_t})=|h_t'(e^{iW_t})|e^{i\wt W_t -i W_t}$, this implies that
\begin{align}\label{eq1}
\left(\frac{d}{dt} h_t\right) (z) = h_t(z) \frac{e^{i\wt W_t}+h_t(z)}{e^{i\wt W_t}-h_t(z)} |h_t'(e^{iW_t})|^2-h_t'(z) z \frac{e^{iW_t}+z}{e^{i W_t}-z}.
\end{align}
In~\eqref{eq1}, we want to plug in $z=e^{iW_t}$. The denominators of the right hand-side of~\eqref{eq1} are zero at $z=e^{iW_t}$, but~\eqref{eq1} still has a limit as $z$ tends to $e^{iW_t}$.
Let $\eps=e^{iW_t}-z$, then $h_t(z)=e^{i\wt W_t}-h_t'(e^{iW_t})\eps$ and $h_t'(z)=h_t'(e^{iW_t})-h_t''(e^{iW_t})\eps$. By looking at the Taylor development of the right hand-side of~\eqref{eq1} as $\eps$ tends to $0$, we get that
\begin{align*}
\left(\frac{d}{dt} h_t\right)(e^{iW_t})=-3 |h_t'(e^{iW_t})|^2 e^{i\wt W_t}+3h_t'(e^{iW_t})e^{iW_t}+3h_t''(e^{iW_t})e^{2iW_t}.
\end{align*}
Therefore
\begin{align*}
d h_t(e^{iW_t})=& (\partial_t h_t)(e^{iW_t}) dt +h_t'(e^{iW_t}) e^{iW_t} i dW_t -\frac{\kappa}{2} \left[h_t''(e^{iW_t})e^{2i W_t}
+h_t'(e^{iW_t})e^{iW_t}\right] dt\\
=& h_t'(e^{iW_t}) e^{iW_t} i dW_t -3\frac{h_t'(e^{iW_t})^2}{h_t(e^{iW_t})}e^{2iW_t} dt +(3-\frac{\kappa}{2}) \left[ h_t'(e^{iW_t})e^{iW_t}+h_t''(e^{iW_t})e^{2iW_t} \right] dt.
\end{align*}
Then
\begin{align}
\notag d\phi_t(W_t)=&-i \frac{dh_t(e^{iW_t})}{h_t(e^{iW_t})}+i\frac{d\langle h_t(e^{iW_t}) \rangle}{2h_t(e^{iW_t})^2}\\
\notag =&\frac{h_t'(e^{iW_t})}{h_t(e^{iW_t})}e^{iW_t} dW_t +3i  \frac{|h_t'(e^{iW_t})|^2}{h_t(e^{iW_t})} e^{i\wt W_t}dt-3i\frac{h_t'(e^{iW_t})}{h_t(e^{iW_t})}e^{iW_t} dt\\
\notag &-i3\frac{h_t''(e^{iW_t})}{h_t(W_t)}e^{2iW_t} dt+i\frac{\kappa}{2} \left[\frac{h_t''(e^{iW_t})}{h_t(e^{iW_t})}e^{2i W_t}
+\frac{h_t'(e^{iW_t})}{h_t(e^{iW_t})}e^{iW_t}\right] dt -i\frac{\kappa}{2}\frac{h'(e^{iW_t})^2}{h(e^{iW_t})^2}e^{2iW_t}dt\\
\label{phi_t(W_t)} =&\phi_t'(W_t) d W_t + (\frac{\kappa}{2}-3) \phi_t''(W_t) dt.
\end{align}
Note that
\begin{align*}
&\phi_t'(z)=\frac{h_t'(e^{iz})}{h_t(e^{iz})}e^{iz}, \quad \phi_t''(z)=\left[ \frac{h_t''(e^{iz})}{h_t(e^{iz})}-\frac{h_t'(e^{iz})^2}{h_t(e^{iz})^2} \right]e^{2iz}i+\frac{h_t'(e^{iz})}{h_t(e^{iz})}e^{iz}i,\\
&\phi_t'''(z)=\left[-\frac{h_t'''(e^{iz})}{h_t(e^{iz})}+3\frac{h_t'(e^{iz}) h_t''(e^{iz})}{h_t(e^{iz})^2}-2\frac{h_t'(e^{iz})^3}{h_t(e^{iz})^3} \right]e^{3iz} +3\left[\frac{h_t'(e^{iz})^2}{h_t(e^{iz})^2} -\frac{h_t''(e^{iz})}{h_t(e^{iz})}\right]e^{2iz}-\frac{h_t'(e^{iz})}{h_t(e^{iz})} e^{iz}.
\end{align*}
Differentiating (\ref{eq1}), we get
\begin{align*}
\left(\frac{d}{dt} h'_t\right)(z)=& h_t'(z)  \frac{e^{i\wt W_t}+h_t(z)}{e^{i\wt W_t}-h_t(z)} |h_t'(e^{iW_t})|^2+ h_t(z) \frac{2e^{i\wt W_t} h_t'(z)}{(e^{i\wt W_t}-h_t(z))^2} |h_t'(e^{iW_t})|^2\\
& -h_t''(z)z \frac{e^{iW_t}+z}{e^{i W_t}-z}-h_t'(z)\frac{e^{iW_t}+z}{e^{i W_t}-z}-h_t'(z)z\frac{2e^{iW_t}}{(e^{iW_t}-z)^2}.
\end{align*}
Therefore
\begin{align*}
\left(\frac{d}{dt} h'_t\right)(e^{iW_t})=& |h_t'(e^{iW_t})|^2\left( -\frac{h_t''(e^{iW_t})}{h_t'(e^{iW_t})}e^{i\wt W_t}-h_t'(e^{iW_t})\right)
+2e^{i\wt W_t} |h_t'(e^{iW_t})|^2\Bigg( \frac{h_t'''(e^{iW_t})}{2h_t'(e^{iW_t})^2}e^{i\wt W_t}+\frac{3h_t''(e^{iW_t})}{2h_t'(e^{iW_t})}\\
&-\frac{h_t'''(e^{iW_t})}{3h_t'(e^{iW_t})^2}e^{i\wt W_t}
+\frac{3h_t''(e^{iW_t})^2}{4h_t'(e^{iW_t})^3}e^{i\wt W_t}-\frac{h_t''(e^{iW_t})}{h_t'(e^{iW_t})}-\frac{h_t''(e^{iW_t})^2}{h_t'(e^{iW_t})^3}e^{i\wt W_t} \Bigg)+3h_t''(e^{iW_t}) e^{iW_t}\\
&+2h_t'''(e^{iW_t}) e^{2iW_t}+h_t'(e^{iW_t})+2h_t''(e^{iW_t}) e^{iW_t}-2e^{iW_t} \left(h_t''(e^{iW_t})+\frac12 h_t'''(e^{iW_t}) e^{iW_t}\right)\\
=&-e^{2iW_t}\frac{h_t'(e^{iW_t})^3}{h_t(e^{iW_t})^2}+3h_t''(e^{iW_t}) e^{iW_t} +\frac43 e^{2iW_t}h_t'''(e^{iW_t})
+h_t'(e^{iW_t})-\frac12 e^{2iW_t} \frac{h_t''(e^{iW_t})^2}{h_t'(e^{iW_t})}.
\end{align*}
Then
\begin{align*}
d h_t'(e^{iW_t})=& (\partial_t h_t')(e^{iW_t}) dt+h_t''(e^{iW_t}) e^{iW_t} i dW_t -\frac{\kappa}{2}\left( h_t'''(e^{iW_t})e^{2iW_t} + h_t''(e^{iW_t})e^{iW_t}\right)dt\\
=& h_t''(e^{iW_t}) e^{iW_t} i dW_t-e^{2iW_t}\frac{h_t'(e^{iW_t})^3}{h_t(e^{iW_t})^2} dt +(3-\frac{\kappa}{2}) h_t''(e^{iW_t}) e^{iW_t}dt \\
&+(\frac43-\frac{\kappa}{2}) e^{2iW_t}h_t'''(e^{iW_t})dt
+h_t'(e^{iW_t}) dt -\frac12 e^{2iW_t} \frac{h_t''(e^{iW_t})^2}{h_t'(e^{iW_t})} dt.
\end{align*}
Consequently
\begin{align*}
d\phi_t'(W_t)=&\frac{dh_t'(e^{iW_t})}{h_t(e^{iW_t})}e^{iW_t} dt-\frac{h_t'(e^{iW_t}) dh_t(e^{iW_t})}{h_t(e^{iW_t})^2}e^{iW_t}-\kappa \frac{h_t'(e^{iW_t})^3}{h_t(e^{iW_t})^3}e^{3iW_t}dt+\frac{h_t'(e^{iW_t})}{h_t(e^{iW_t})} e^{iW_t} i dW_t\\
&-\frac{\kappa}{2} \frac{h_t'(e^{iW_t})}{h_t(e^{iW_t})} e^{iW_t} dt +\kappa \frac{h_t'(e^{iW_t}) h_t''(e^{iW_t})}{h_t(e^{iW_t})^2} e^{3iW_t} dt -\kappa \frac{h_t''(e^{iW_t})}{h_t(e^{iW_t})}e^{2iW_t} dt +\kappa \frac{h_t'(e^{iW_t})^2}{h_t(e^{iW_t})^2} e^{2i W_t} dt\\
=&\phi_t''(W_t) dW_t- (\frac{4}{3}-\frac{\kappa}{2}) \phi_t'''(W_t)dt +\frac12 \frac{\phi_t''(W_t)^2}{\phi_t'(W_t)} dt+\frac16(\phi_t'(W_t)-\phi_t'(W_t)^3) dt.
\end{align*}

\subsection{Change of driving function under conformal map}
\begin{lemma}\label{lem:change_measure_conformal}
Let $\gamma$ be a radial hSLE$_\kappa(\mu,\nu)$ in $\Ub$. For any $A\in\Qc$, let $\wt\gamma$ be the image by $f_A^{-1}$ of $\gamma$. Then up to reparametrization, $\wt \gamma$ is a radial SLE with driving function
\begin{align}\label{eq:df_t_gamma}
dW_t= \sqrt{\kappa}d B_t+\!\left( \frac{6-\kappa}{2}\frac{\phi_t''(W_t)}{\phi_t'(W_t)} +\frac{\kappa}{2} \phi_t'(W_t) \frac{G'(\nu_t)}{G(\nu_t)} \right) dt.
\end{align}
\end{lemma}
\begin{proof}
It is equivalent to prove the statement in the other direction: Suppose that $\wt \gamma$ is a radial SLE in $\Ub$ driven by~\eqref{eq:df_t_gamma}, we will show that $f_A(\wt\gamma)$ is distributed as a radial hSLE$_\kappa(\mu,\nu)$ in $\Ub$, up to reparametrization.
Let $\gamma$ be $f_A(\wt\gamma)$. Then the driving function of $\gamma$ is given by $\phi_t(W_t)$ where $W_t$ is given by~\eqref{eq:df_t_gamma}.
By~\eqref{phi_t(W_t)}, we get that
\begin{align*}
d\phi_t(W_t)=\phi_t'(W_t)\sqrt{\kappa} dB_t +\frac{\kappa}{2} \phi_t'(W_t)^2 \frac{G'(\nu_t)}{G(\nu_t)} dt.
\end{align*}
Making the time change $s(t)=\int_0^t \phi_r'(W_r)^2 dr$, $\gamma$ parametrized by $s$ is indeed a radial hSLE$_\kappa(\mu,\nu)$ in $\Ub$.
\end{proof}

\subsection{Proof of Lemma~\ref{lem:martingale}}\label{A3}
We are now going to prove Lemma~\ref{lem:martingale}. Having obtained the intermediate results in Appendix~\ref{A:ito}, we will make straightforward application of It\^o formula in this section. 
We believe that it is not possible to bypass this computation (also see Remark~\ref{remark}). 
In order for $(M_t)$ to be a martingale, as we will show in the computation below, $G$ has to satisfy~\eqref{diff-G}. In reality, we have obtained $G$ by solving the second order linear differential equation~\eqref{diff-G} (with additional conditions), but we have decided to postpone this computational part to the very end.
\smallskip

We consider here a radial hSLE$_\kappa(\mu, \nu)$ with driving function $(W_t)$ given by~\eqref{driving-function}.
We will prove that the following is a local martingale:
\begin{align}\label{mart}
M_t=&|h_t'(0)|^{e_1} \, |h_t'(e^{iW_t})|^{e_2} \, |h_t'(e^{iV_t})|^{e_3} \, \frac{G(\nu_t)}{G(\theta_t)}\, \exp\left( \int_0^t -\frac{c(\kappa)}{6} |Sh_s(W_s)| ds \right),
\end{align}
where $e_1=2\mu$, $e_2=(6-\kappa)/(2\kappa)$ and $e_3=\nu$. We will perform It\^o calculus and show that the drift term of $d M_t$ is zero.

First note that
\begin{align}\label{M1}
\frac{d M_t}{M_t}=d \ln M_t +\frac{d\langle M\rangle_t}{2 M_t^2}
\end{align}
and
\begin{equation}\label{M2}
\begin{split}
d \ln M_t=&e_1 d\ln |h_t'(0)| +e_2 d \ln  |h_t'(e^{iW_t})| +e_3 d \ln  |h_t'(e^{iV_t})|\\
& + d\ln G(\nu_t) - d\ln G(\theta_t)-\frac{c(\kappa)}{6} |Sh_t(W_t)|dt.
\end{split}
\end{equation}
For ease, we introduce the notations
\begin{align}\label{x123}
X_1=\phi_t'(W_t),\, X_2=\phi_t''(W_t),\, X_3=\phi_t'''(W_t), \, Y_1=\phi_t'(V_t).
\end{align}
Using the results of Section~\ref{A:ito}, we have
\begin{align*}
& d\phi_t(W_t)=X_1 \sqrt{\kappa}d B_t + X_1 \frac{\kappa}{2} \frac{G'(\theta_t)}{G(\theta_t)} dt + (\kappa/2-3) X_2 dt\\
& d \phi_t(V_t)=-X_1^2 \cot(\nu_t) dt\\
& d \phi_t'(W_t) = X_2 \sqrt{\kappa} dB_t +X_2 \frac{\kappa}{2} \frac{G'(\theta_t)}{G(\theta_t)} dt +\left(\frac{X_2^2}{2X_1}+\frac{X_1-X_1^3}{6} \right) dt +(\frac{\kappa}{2}-\frac43) X_3\\
& d \phi_t'(V_t) =-\frac12 X_1^2 Y_1 \frac{1}{\sin(\nu_t)^2} dt +\frac12 Y_1\frac{1}{\sin(\theta_t)^2} dt \\
& d \theta_t = \frac12 \sqrt{\kappa} dB_t +\frac{\kappa}{4 }\frac{G'(\theta_t)}{G(\theta_t)} dt +\frac12 \cot(\theta_t) dt\\
& d\nu_t =\frac{X_1}{2}\sqrt{\kappa} dB_t +\frac{X_1}{2} \frac{\kappa}{2} \frac{G'(\theta_t)}{G(\theta_t)} dt +\left(\frac{\kappa}{4}-\frac32\right)X_2 dt +\frac12 X_1^2\cot(\nu_t) dt
\end{align*}
Therefore
\begin{align*}
&&d \ln  |h_t'(e^{iW_t})| &= \frac{dX_1}{X_1}-\frac{d\langle X_1\rangle_t}{2X_1^2}  \\
&&&= \frac{X_2}{X_1} \sqrt{\kappa} dB_t +  \frac{X_2}{X_1} \frac{\kappa}{2} \frac{G'(\theta_t)}{G(\theta_t)} dt+ \left(\frac{X_2^2}{2X_1^2}+(\frac{\kappa}{2}-\frac43)\frac{X_3}{X_1}+\frac{1-X_1^2}{6} \right) dt - \frac{\kappa X_2^2}{2 X_1^2} dt\\
&& d \ln  |h_t'(e^{iV_t})| &=-\frac12 X_1^2  \frac{1}{\sin(\nu_t)^2} dt +\frac12 \frac{1}{\sin(\theta_t)^2} dt \\
&& d \ln G(\nu_t) &= \frac{G'(\nu_t)}{G(\nu_t)} d \nu_t +\frac12\left(\frac{G''(\nu_t)}{G(\nu_t)}-\frac{G'(\nu_t)^2}{G(\nu_t)^2}\right) \frac{X_1^2 \kappa}{4} dt \\
&& d \ln G(\theta_t) &= \frac{G'(\theta_t)}{G(\theta_t)} d \theta_t +\frac12\left(\frac{G''(\theta_t)}{G(\theta_t)}-\frac{G'(\theta_t)^2}{G(\theta_t)^2}\right) \frac{\kappa}{4} dt.
\end{align*}
Combined with~\eqref{M1} and~\eqref{M2}, it follows that the local martingale part of $d M_t$ is equal to
\begin{align}\label{lm}
M_t \left( e_2\frac{X_2}{X_1} +\frac{G'(\nu_t)}{G(\nu_t)}\frac{X_1}{2}-\frac{G'(\theta_t)}{G(\theta_t)}\frac12  \right) \sqrt{\kappa} dB_t.
\end{align}
Hence we have
\begin{align*}
\frac{d\langle M\rangle_t}{M_t^2}=\left( e_2\frac{X_2}{X_1} +\frac{G'(\nu_t)}{G(\nu_t)}\frac{X_1}{2}-\frac{G'(\theta_t)}{G(\theta_t)}\frac12  \right)^2 \kappa dt.
\end{align*}
Note also that by~\eqref{eq:Sch}, we have
\begin{align*}
|Sh_t(W_t)|=X_3/X_1-(3X_2^2)/(2X_1^2).
\end{align*}
Putting the results above back into~\eqref{M1} and~\eqref{M2}, we are now able to write down the drift term of $d M_t/ M_t$, which is the sum of all the terms listed below:
\begin{align}
\label{E1}
X_1^2 \quad \times  & \quad e_1 -\frac{e_2}{6}-\frac{c(\kappa)}{12}-\frac{e_3}{2\sin(\nu_t)^2}+\frac{G'(\nu_t)}{G(\nu_t)}\frac{\cot(\nu_t)}{2} +\frac{\kappa}{8} \frac{G''(\nu_t)}{G(\nu_t)} \\
\notag
X_2/X_1 \quad \times & \quad e_2 (\frac{\kappa}{2}-\frac{\kappa}{2})\frac{G'(\theta_t)}{G(\theta_t)}=0\\
\notag
X_2^2/X_1^2\quad \times  & \quad e_2 (1-\kappa)/2 +e_2^2\kappa/2 + c(\kappa)/4 =0 \\
\notag
X_3/X_1 \quad \times & \quad e_2(\frac{\kappa}{2}-\frac{4}{3})-c(\kappa)/6=0\\
\notag
X_2\quad \times& \quad \left(\frac{\kappa}{4}-\frac{3}{2}+\frac{e_2\kappa}{2}\right) \frac{G'(\nu_t)}{G(\nu_t)}=0 \\
\notag
X_1\quad \times & \quad \left( \frac{\kappa}{4}-\frac{\kappa}{4} \right) \frac{G'(\nu_t)}{G(\nu_t)} \frac{G'(\theta_t)}{G(\theta_t)}=0\\
\label{E2}
1 \quad \times  &\quad -e_1+\frac{e_2}{6}+\frac{c(\kappa)}{12}+\frac{e_3}{2\sin(\theta_t)^2}- \frac{\cot(\theta_t) G'(\theta_t)}{2 G(\theta_t)} -\frac{\kappa}{8}  \frac{G''(\theta_t)}{G(\theta_t)}.
\end{align}
Note that $G$ is solution to the differential equation~\eqref{diff-G}, hence the terms~\eqref{E1} and~\eqref{E2} are also equal to zero.
Therefore, the drift term of $dM_t/M_t$ is zero, hence $M$ is indeed a local martingale. 
This completes the proof of Lemma~\ref{lem:martingale}.

Finally, let us note the following result which is used in the proof of Proposition~\ref{prop-restriction}: By~\eqref{lm} and the fact that $M$ is a local martingale, we have
\begin{align}\label{eq:dM/M}
\frac{dM_t}{M_t}=\!\left(\frac{6-\kappa}{2\kappa} \frac{\phi_t''(W_t)}{\phi_t'(W_t)} +\frac{\phi_t'(W_t)}{2} \frac{G'(\nu_t)}{G(\nu_t)} -\frac12 \frac{G'(\theta_t)}{G(\theta_t)} \right) \sqrt{\kappa} dB_t.
\end{align}

\bibliographystyle{alpha}
\bibliography{cr} 

\medbreak

{\small CNRS and Laboratoire de Math\'ematiques d'Orsay, Universit\'e Paris-Saclay
\smallskip

E-mail: wei.qian@universite-paris-saclay.fr}

\end{document}